\documentclass[11pt]{article}
\usepackage{amsmath,amsthm,amssymb,amsfonts}
\usepackage{enumerate}
\usepackage{mathrsfs}
\usepackage[cmtip,all]{xy}

\normalsize

\newtheoremstyle{theoremstyle}
  {10pt}      
  {5pt}       
  {\itshape}  
  {}          
  {\bfseries} 
  {:}         
  {.5em}      
  {}          

\newtheoremstyle{examplestyle}
  {10pt}      
  {5pt}       
  {}          
  {}          
  {\bfseries} 
  {:}         
  {.5em}      
  {}          

\theoremstyle{theoremstyle}
\newtheorem{theorem}{Theorem}[section]
\newtheorem*{theorem*}{Theorem}
\newtheorem{lemma}[theorem]{Lemma}
\newtheorem{proposition}[theorem]{Proposition}
\newtheorem*{proposition*}{Proposition}
\newtheorem{corollary}[theorem]{Corollary}
\newtheorem*{corollary*}{Corollary}

\newtheorem{definition}[theorem]{Definition}
\newtheorem{definition*}{Definition}
\newtheorem{remark}[theorem]{Remark}
\newtheorem{remark*}{Remark}
\newtheorem{problem}[theorem]{Problem}

\newcommand{\A}{\mathbf{A}}
\newcommand{\G}{\mathbf{G}_{\mathfrak{m}}}
\newcommand{\AAA}{\mathsf{A}}
\newcommand{\B}{\mathsf{B}}

\newcommand{\E}{\mathsf{E}}
\newcommand{\F}{\mathsf{F}}
\newcommand{\fff}{\mathsf{f}}
\newcommand{\GG}{\mathsf{G}}
\newcommand{\caE}{{\mathcal E}}
\newcommand{\caK}{{\mathcal K}}
\newcommand{\MM}{{\mathscr M}}
\newcommand{\ZZ}{{\mathscr Z}}
\newcommand{\caO}{{\mathcal O}}
\newcommand{\x}{\mathsf{x}}
\newcommand{\PP}{\mathbf{P}}
\newcommand{\CP}{\mathbf{CP}}
\newcommand{\CC}{\mathbf{C}}
\newcommand{\BG}{\mathsf{B}\mathbf{G}_{\mathfrak{m}}}
\newcommand{\BS}{\mathsf{B}S^{1}}
\newcommand{\BGL}{\mathsf{B}\mathbf{GL}}
\newcommand{\KGL}{\mathsf{KGL}}
\newcommand{\MGL}{\mathsf{MGL}}
\newcommand{\kgl}{\mathsf{kgl}}
\newcommand{\Tor}{\mathsf{Tor}}
\newcommand{\MZ}{\mathsf{M}\mathbf{Z}}
\newcommand{\MQ}{\mathsf{M}\mathbf{Q}}
\newcommand{\Ch}{\mathsf{Ch}}
\newcommand{\PM}{\mathsf{PM}}
\newcommand{\KU}{\mathsf{KU}}
\newcommand{\MU}{\mathsf{MU}}
\newcommand{\unit}{\mathbf{1}}
\newcommand{\Z}{\mathbf{Z}}
\newcommand{\Q}{\mathbf{Q}}
\newcommand{\QQ}{\mathsf{Q}}
\newcommand{\Sq}{\mathsf{Sq}}
\newcommand{\X}{\mathscr{X}}
\newcommand{\Y}{\mathscr{Y}}
\newcommand{\QQQ}{\mathscr{Q}}

\newcommand{\Spc}{\mathbf{Spc}}
\newcommand{\Sm}{\mathbf{Sm}}
\newcommand{\SH}{\mathbf{SH}}
\newcommand{\Ho}{\mathbf{Ho}}
\newcommand{\HH}{\mathbf{H}}
\newcommand{\Thom}{\mathsf{Th}}

\newcommand{\Mod}{\mathbf{Mod}}
\newcommand{\Pic}{\mathbf{Pic}}
\newcommand{\bL}{\mathbf{L}}
\newcommand{\Nis}{\mathbf{Nis}}
\newcommand{\Hom}{\mathrm{Hom}}
\newcommand{\uHom}{\underline{\mathrm{Hom}}}
\newcommand{\eff}{\mathbf{eff}}
\newcommand{\Veff}{\mathbf{Veff}}
\newcommand{\colim}{\mathrm{colim}}
\newcommand{\hocolim}{\mathrm{hocolim}}
\newcommand{\proj}{\mathbf{proj}}
\newcommand{\Spt}{\mathsf{Spt}}

\newcommand{\rto}{\rightarrow}
\newcommand{\lto}{\leftarrow}

\title{{\bf Motivic twisted K-theory}}
\author{Markus Spitzweck, Paul Arne {\O}stv{\ae}r}
\date{August 1, 2010}
\begin{document}
\maketitle
\begin{abstract}
This paper sets out basic properties of motivic twisted K-theory with respect to degree three motivic cohomology classes of weight one.
Motivic twisted K-theory is defined in terms of such motivic cohomology classes by taking pullbacks along the universal principal $\BG$-bundle for the 
classifying space of the multiplicative group scheme.
We show a K{\"u}nneth isomorphism for homological motivic twisted K-groups computing the latter as a tensor product of K-groups over the K-theory of $\BG$.
The proof employs an Adams Hopf algebroid and a tri-graded Tor-spectral sequence for motivic twisted K-theory.
By adopting the notion of an $E_{\infty}$-ring spectrum to the motivic homotopy theoretic setting, 
we construct spectral sequences relating motivic (co)homology groups to twisted K-groups.
It generalizes various spectral sequences computing the algebraic K-groups of schemes over fields.
Moreover, 
we construct a Chern character between motivic twisted K-theory and twisted periodized rational motivic cohomology, 
and show that it is a rational isomorphism.
The paper includes a discussion of some open problems.
\end{abstract}
{\small{\tableofcontents}}
\newpage

\section{Motivation}
\label{section:motivation}

Topological K-theory has many variants which have all been developed and exploited for geometric purposes. 
Twisted K-theory or ``K-theory with coefficients'' was introduced by Donovan and Karoubi in \cite{Karoubi-Donovan} using Wall's graded Brauer group.
More general twistings of K-theory arise from automorphisms of its classifying space of Fredholm operators on an infinite dimensional separable complex Hilbert space.
Of particular geometric interest are twistings given by integral $3$-dimensional cohomology classes.
The subject was further developed in the direction of analysis by Rosenberg in \cite{Rosenberg}.

Twisted K-theory resurfaced in the late 1990's with Witten's work on classification of D-branes charges in type II string theory \cite{Witten}.
Fruitful interactions between algebraic topology and physics afforded by twisted K-theory continues today,
see e.g.~\cite{ArbeitMFO}, \cite{Atiyah-Segal1}, \cite{Atiyah-Segal2}, \cite{BCMMS} and \cite{TXLG}.
The work of Freed, Hopkins and Teleman \cite{FHT} relates the twisted equivariant K-theory of a compact Lie group $G$ to the ``Verlinde ring'' of its loop group.
For a recent survey of twisted K-theory we refer to \cite{Karoubi}. 

We are interested in twistings of the motivic K-theory spectrum $\KGL$ in the algebro-geometric context of motivic homotopy theory  \cite{DLORV}, \cite{Voevodsky-icm}.
Over the complex numbers $\CC$, 
or more generally any field with a complex embedding, 
our motivic twisted K-theory specializes to twisted K-theory under realization of complex points.
The idea of twisting (co)homology theories have been used to great effect in topology. 
A classical example is cohomology with local coefficients, 
which can be used to formulate Poincar{\'e} duality and the Thom isomorphism for unorientable manifolds. 
Analogous motivic results are subject to future work.

In \cite{ABG}, 
Ando, Blumberg and Gepner use the formalism of $\infty$-categories in order to construct twisted forms of multiplicative generalized (co)homology theories, 
and in \cite{MS}, May and Sigurdsson employ parametrized spectra to the same end.
Their setups suggest analogous algebro-geometric generalizations.

\section{Main results}
\label{section:mainresults}

The isomorphism classes of principal $\BS$-bundles over a topological space $X$ identifies canonically with the homotopy classes of maps from $X$ to the Eilenberg-MacLane 
space $K(\Z,3)$.
The second delooping $\B\BS$ of the circle gives a concrete model for $K(\Z,3)$. 
We begin the paper by considering the analogous setup in motivic homotopy theory.

Let $\X$ be a motivic space and $\G$ the multiplicative group scheme over a noetherian base scheme $S$ of finite dimension (usually left implicit in the notation).

For any map $\tau\colon\X\to\B\BG$ define $\X^{\tau}$ as the homotopy pullback along $\ast\rto\B\BG$ 
(which can be thought of as a universal principal $\BG$-bundle):
\begin{equation*}
\xymatrix{
\X^{\tau}\ar[r]\ar[d] & 
\ast \ar[d] \\
\X\ar[r]^-{\tau}  & \B\BG
}
\end{equation*}
With this definition there is a naturally induced module map
\begin{equation*}
\PP^{\infty}\times\X^{\tau}
\rto
\X^{\tau}.
\end{equation*}
Here we use implicitly the motivic weak equivalence between $\BG$ and the infinite projective space $\PP^{\infty}$.
By passing to motivic suspension spectra we get a naturally induced map
\begin{equation*}
\Sigma^{\infty}\PP^{\infty}_{+}\wedge\Sigma^{\infty}\X^{\tau}_{+}
\rto
\Sigma^{\infty}\X^{\tau}_{+}
\end{equation*}
displaying $\Sigma^{\infty}\X^{\tau}_{+}$ as a module over the motivic ring spectrum $\Sigma^{\infty}\PP^{\infty}_{+}$.
(We defer the somewhat technical definition of this module structure to Section \ref{section:groupactionsandquotientspaces}.)

When $S$ is a smooth scheme over a field we can identify the homotopy classes of maps from $\X$ to $\B\BG$ with the third integral motivic cohomology group $\MZ^{3,1}(\X)$ of weight one.
This group is in fact trivial for smooth schemes of finite type over $S$ by \cite[Corollary 3.2.1]{Suslin-Voevodsky}.
On the other hand, 
it is nontrivial for motivic spheres.

Denote by $\BGL$ the classifying space of the infinite Grassmannian over $S$.
Suppose $f_{\zeta}\colon\X\rto\PP^{\infty}$ and $f_{\xi}\colon\X\rto\Z\times\BGL$ represent $\zeta\in\MZ^{2,1}(\X)$ and $\xi\in\KGL_{\ast}(\X)$, respectively. 
Then the composite map
\begin{equation*}
\X
\overset{\Delta}{\rto}
\X\times\X
\overset{f_{\zeta}\times f_{\xi}}{\rto}
\PP^{\infty}\times (\Z\times\BGL)
\rto
\Z\times\BGL
\end{equation*}
represents an element $\xi\otimes\zeta\in\KGL_{\ast}(\X)$.
The above defines the action of the Picard group on the K-theory ring of $\X$.
On the level of motivic spectra there exists a corresponding composite map 
\begin{equation*}
\Sigma^{\infty}\PP^{\infty}_{+}\wedge\KGL
\rto
\KGL\wedge\KGL
\rto
\KGL,
\end{equation*}
where the second map is the ring multiplication on $\KGL$.
The first map is obtained via adjointness from the multiplicative map $\BG\rto\{1\}\times\BGL$ that sends a line bundle represented by a map into $\PP^{\infty}$ to its class in the Grothendieck group of all vector bundles \cite[(3)]{Spitzweck-Ostvar}. 
Thus the motivic K-theory spectrum $\KGL$ is a module over $\Sigma^{\infty}\PP^{\infty}_{+}$.

We are now ready to define our main objects of study in this paper.
The distinction between homological and cohomological versions of motivic twisted K-theory is rooted in standard nomenclature for twisted K-theory.
\begin{definition}
For $\tau\colon\X\rto\B\BG$ define the motivic twisted 
\begin{itemize}
\item homological K-theory of $\tau$ by $\KGL^{\tau}\equiv\Sigma^{\infty}\X^{\tau}_{+}\wedge_{\Sigma^{\infty}\PP^{\infty}_{+}}\KGL$.
\item cohomological K-theory of $\tau$ by $\KGL_{\tau}\equiv\uHom_{\Sigma^\infty \PP^\infty_+}(\Sigma^\infty \X^\tau_+,\KGL)$.
\end{itemize}
\end{definition}

The smash product in the definition of $\KGL^{\tau}$ is derived in the sense that it is formed in the homotopy category of highly structured modules over $\Sigma^{\infty}\PP^{\infty}_{+}$.
In order to make sense of the derived smash product, 
we implicitly use a closed symmetric monoidal model for the motivic stable homotopy category, 
see Jardine's work on motivic symmetric spectra \cite{Jardine}, for example.
Later in the paper we prove that homotopy equivalent maps from $\X$ to $\B\BG$ give rise to isomorphic motivic twisted K-theories.
In addition, 
the derived style definition of $\KGL^{\tau}$ requires a strict ring model for $\KGL$ as a $\Sigma^{\infty}\PP^{\infty}_{+}$-ring spectrum.
Such a model was only recently constructed in \cite{Rondigs-Spitzweck-Ostvar} using the Bott tower 
\begin{equation}
\label{equation:Botttower}
\Sigma^{\infty}\PP^{\infty}_{+}
\overset{\beta}{\rto}
\Sigma^{-2,-1}\Sigma^{\infty}\PP^{\infty}_{+}
\overset{\Sigma^{-2,-1}\beta}{\rto}
\cdots.
\end{equation}
Similarly,
in the cohomological setup,
the hom-object appearing in the definition of  $\KGL_{\tau}$ is formed in the homotopy category of $\Sigma^{\infty}\PP^{\infty}_{+}$-modules.

An alternate definition of $\KGL^{\tau}$ can be made precise by simply inverting the $(2,1)$ self-map of $\Sigma^{\infty}\X^{\tau}_{+}$ obtained from the Bott map realizing K-theory $\KGL$ 
as the Bott inverted infinite projective space. 
(The Bott map $\beta$ is indeed a $\Sigma^{\infty}\PP^{\infty}_{+}$-module map by construction.) 
Independent proofs of the latter result have appeared in  \cite{Gepner-Snaith} and \cite{Spitzweck-Ostvar}.
For other discussions of the Bott inverted model for K-theory we refer to \cite{nmp-nonregular}, \cite{nmp-exactness}, \cite{Rondigs-Spitzweck-Ostvar}.
Making use of the Bott map provides also an alternate definition of $\KGL_{\tau}$.
We shall be using this viewpoint on a number of occasions in this paper.

Now suppose the twisting class $\tau$ for $\X$ is null and the base scheme $S$ is regular.
In Section \ref{section:groupactionsandquotientspaces} we show that the homotopy fiber $\X^{\tau}$ identifies with the product $\PP^{\infty}\times\X$ and $\KGL^{\tau}$ identifies with the 
smash product $\Sigma^{\infty}\X_{+}\wedge\KGL$.
Accordingly, 
we may view motivic twisted K-theory as a generalization of K-theory.
In the course of the paper we shall work out some of the differences and similarities arising from this generalization, 
and suggest some open problems.

For a general twist $\tau$ it turns out that the motivic twisted K-theory spectra $\KGL^{\tau}$ and $\KGL_{\tau}$ do not acquire ring structures in the motivic stable homotopy category.
In particular, 
the motivic twisted K-groups $\KGL^{\tau}_{\ast}$ do not form a ring in general.
The lack of a product structure tends to complicate computations.
On the other hand we establish two far more powerful tools for performing computations.
First we prove a K{\"u}nneth isomorphism for homological motivic twisted K-groups and second we construct spectral sequences relating motivic (co)homology to motivic twisted K-groups.

By applying the Tor-spectral sequence in \cite[Proposition 7.7]{Dugger-Isaksen} to the commutative motivic ring spectrum $\Sigma^{\infty}\PP^{\infty}_{+}\wedge\KGL$ and its modules 
$\KGL$ and $\Sigma^{\infty}\X^{\tau}_{+}\wedge\KGL$ we arrive at the strongly convergent tri-graded spectral sequence
\begin{equation}
\label{equation:torss}
\Tor^{\KGL_{\ast}(\PP^{\infty})}_{a,(b,c)}(\KGL_{\ast}(\X^{\tau}),\KGL_{\ast})
\Rightarrow
\KGL_{a+b,c}^{\tau}(\X).
\end{equation}
Here we should infer that $\Sigma^{\infty}\PP^{\infty}_{+}\wedge\KGL$ and $\KGL$ are stably cellular motivic ring spectra, 
a.k.a.~``Tate spectra'' \cite{nmp-nonregular}.
To begin with, 
the motivic K-theory spectrum $\KGL$ is stably cellular \cite[Theorem 6.2]{Dugger-Isaksen}.
The suspension spectrum of the pointed infinite projective space is also stably cellular by \cite[Propositions 2.13, 2.17, Lemma 3.1]{Dugger-Isaksen} since any filtered colimit of stably cellular 
motivic spaces is stably cellular \cite[Definition 2.1(3)]{Dugger-Isaksen}.
Hence the smash product $\Sigma^{\infty}\PP^{\infty}_{+}\wedge\KGL$ is cellular.
(Although the case of fields is emphasized in \cite{Dugger-Isaksen} the results we employ from loc.~cit.~hold over arbitrary noetherian base schemes 
of finite dimension.)
\begin{theorem}
\label{theorem:edgemapisomorphism}
The edge homomorphism in the Tor-spectral sequence (\ref{equation:torss}) induces a natural isomorphism
\begin{equation*}
\KGL_{\ast}^{\tau}(\X)
\cong
\KGL_{\ast}(\X^{\tau})\otimes_{\KGL_{\ast}(\PP^{\infty})}\KGL_{\ast}.
\end{equation*}
\end{theorem}
This is the motivic analogue of the corresponding topological result shown in \cite{Khorami}.
Theorem \ref{theorem:edgemapisomorphism} follows from (\ref{equation:torss}) by proving the Tor-group $\Tor^{\KGL_{\ast}(\PP^{\infty})}_{a,(b,c)}(\KGL_{\ast},\KGL_{\ast}(\X^{\tau}))$ is trivial for $a>0$.
It is worthwhile to point out that $\KGL_{\ast}$ is not a flat $\KGL_{\ast}(\PP^{\infty})$-module, 
i.e.~the vanishing result for the Tor-groups does not hold for an ``obvious''  reason.
Our proof of the vanishing employs flatness of the ring map $\KGL_{\ast}(\PP^{\infty})\rto\KGL_{\ast}\KGL$ and the homotopy theory of Hopf algebroids.
More precisely,
it is shown that the composite map $\KGL_{\ast}(\PP^{\infty})\rto\KGL_{\ast}\KGL\rto\KGL_{\ast}$ satisfies the Landweber exactness criterion relative to the Hopf algebroid  
$(\KGL_{\ast}(\PP^{\infty}),\KGL_{\ast}\KGL\otimes_{\KGL_{\ast}}\KGL_{\ast}(\PP^{\infty}))$. 
Furthermore, 
$\KGL_{\ast}(\X^{\tau})$ is a comodule over the same Hopf algebroid, 
and the Tor-groups are computed by a cofibrant replacement in the projective model structure on the category of unbounded chain complexes of such comodules.  
A crux input is that the Hopf algebroid in question is an ``Adams Hopf algebroid.''
This notion is recalled in Section \ref{section:hopfalgebroids} together with some background from stable homotopy theory.
Combining these facts we show the desired vanishing of the Tor-groups in positive degrees.
Our use of the model structure circumvents a more explicit construction employed in the topological situation \cite{Khorami}.

Algebraic K-theory is closely related to motivic cohomology and more classically to higher Chow groups via Chern characters.
We shall briefly examine a Chern character for motivic twisted K-theory with target twisted periodized rational motivic cohomology 
\begin{equation*}
\Ch^{\tau}
\colon
\KGL^{\tau}
\rto
\PM^{\tau}\Q.
\end{equation*}
The construction of $\Ch^{\tau}$ follows the setup for the Chern character for $\KGL$ in \cite{nmp-exactness}.
As it turns out, 
the rationalization of $\Ch^{\tau}$ is an isomorphism under a mild assumption on the base scheme originating in the work of Cisinski-D{\'e}glise
\cite{Cisinski-Deglise}.
(We leave the formulation of the corresponding result for $\KGL_{\tau}$ to the main body of the paper.)
\begin{theorem}
\label{theorem:chernisomorphism}
For geometrically unibranched excellent base schemes the rational Chern character
\begin{equation*}
\Ch^{\tau}_{\Q}
\colon
\KGL^{\tau}_{\Q}
\rto
\PM^{\tau}\Q.
\end{equation*}
is an isomorphism in the homotopy category of modules over $\Sigma^{\infty}\PP^{\infty}_{+}$.
\end{theorem}

Section \ref{section:proofsofthemainresults} provides streamlined proofs of the results reviewed in the above.
In the same section we work out explicit computations of the motivic twisted K-groups of the motivic $(3,1)$-sphere.
Over finite fields the motivic twisted K-groups in positive degrees $2k-1$ and $2k$ turn out to be finite cyclic groups of the same order.
This amusing computation is closely related to Quillen's computation of the K-groups of finite fields. 
Some of the basic facts concerning flat Adams Hopf algebroids required in Section \ref{section:proofsofthemainresults} are deferred to Section \ref{section:hopfalgebroids}.

In Section \ref{section:spectralsequencesformotivictwisted K-theory} we establish powerful integral relations between motivic (co)homology and motivic twisted K-theory 
in the form of spectral sequences
\begin{equation*}
\MZ_{\ast}(\Sigma^\infty \X_{+} )
\Longrightarrow
\KGL^{\tau}_{\ast}(\X)
\end{equation*}
and
\begin{equation*}
\MZ^{\ast}(\Sigma^\infty \X_{+} )
\Longrightarrow
\KGL_{\tau}^{\ast}(\X).
\end{equation*}
For some closely related papers on spectral sequences computing (non-twisted) K-groups in terms of motivic cohomology groups we refer the reader to \cite{BL}, \cite{FS}, 
\cite{Grayson}, \cite{Levine-slices}, \cite{Suslin} and \cite{Voevodskymssktheory}.
The question of strong convergence of these spectral sequences is a tricky problem.
Our approach involves the very effective motivic stable homotopy category $\SH(S)^{\Veff}$. 
We define it as the smallest full subcategory of the motivic stable homotopy category $\SH(S)$ that contains suspension spectra of smooth schemes of finite type over $S$ 
and is closed under extensions and homotopy colimits.
This is not a triangulated category;
however,
it is a subcategory of the effective motivic stable homotopy category $\SH(S)^{\eff}$.
(In fact, it is the homologically positive part of a $t$-structure on $\SH(S)^{\eff}$.)
The very effective motivic stable homotopy category is of independent interest.
We show that the algebraic cobordism spectrum $\MGL$ lies in $\SH(S)^{\Veff}$.
When $S$ is a field of characteristic zero, 
we show that the connective K-theory spectrum $\kgl$ lies in $\SH(S)^{\Veff}$.
This is a key input for showing strong convergence of the spectral sequences.

The main body of the paper ends in Section \ref{section:furtherquestionsandproblems} with a discussion of open problems. 
In particular, 
we suggest extending the techniques in this paper to the settings of both equivariant K-theory and hermitian K-theory.

For legibility, bigraded motivic homology theories are written with a single grading.
The precise meaning of the gradings should always be clear from the context.

\section{Main definitions and first properties}
\label{section:groupactionsandquotientspaces}
In this section we first put the definitions of the motivic twisted $K$-theory spectra on rigorous grounds.
This part deals with model structures and classifying spaces.
The constructions are rigged so that smashing with the sphere spectrum over $\Sigma^\infty \PP^\infty_+$ yields a useful ``untwisting'' result detailed in Lemma \ref{lemma:P00topoint}.
For algebro-geometric reasons we shall explain below, some of the results require mild restrictions on the base scheme.

Let $\Spc$ be the category of motivic spaces on $\Sm$,
i.e.~simplicial presheaves on the Nisnevich site of smooth schemes of finite type over $S$, 
with the injective motivic model structure.
This model structure satisfies the monoid axiom \cite{schwede-shipley}.
Hence for any monoid $G$ in $\Spc$, 
the category $\Mod(G)$ of $G$-modules acquires a model structure.
For a map $G \to H$ of monoids there is an induced left Quillen functor $\Mod(G) \to \Mod(H)$.
In particular, 
the pushforward of a $G$-module $\X$ along the canonical map $G \to \ast$ is the quotient $\X/G$.
The homotopy quotient is defined similarly by first taking a cofibrant replacement of $\X$ in $\Mod(G)$.

We denote by $\Mod_{\Y}(G)$ the category of $G$-modules in the slice category $\Spc/\Y$ comprised of motivic spaces over a motivic space $\Y$.
It should be noted that $\Mod_{\Y}(G)$ is a model category:
To begin with, 
$\Spc/\Y$ inherits an evident model structure from the motivic model structure on $\Spc$ which turns the pairing 
$\Spc\times\Spc/\Y\rto\Spc/\Y$ sending $(\X,\X^{\prime}\rto\Y)$ to $(\X\times\X^{\prime})\rto\X^{\prime}\rto\Y$ into a Quillen bifunctor.
Moreover, 
every object of $\Spc$ is cofibrant. 
Thus the existence of the model structure on $\Mod_{\Y}(G)$ follows from a relative version of \cite[Theorem 3.1.1]{schwede-shipley}.
For a $G$-module $\X$ and a map of motivic spaces $\Y' \to \Y$, 
note that $\X \in \Mod_{\X/G}(G)$ and there is a pullback functor $\Mod_\Y(G) \to \Mod_{\Y'}(G)$.

In what follows we specialize to the commutative monoid $\BG\simeq\PP^{\infty}$.
As a model for the classifying space $\B \PP^{\infty}$ of $\PP^{\infty}$ we may use the standard bar construction.
Viewing $\ast$ as a $\PP^{\infty}$-module we are entitled to a cofibrant replacement $\QQQ \to \ast$ in $\Mod(\PP^{\infty})$. 
The homotopy quotient $\QQQ/\PP^{\infty}$ gives an alternative model for the classifying space of $\PP^{\infty}$.
In the proof of Lemma \ref{lemma:P00topoint} we find it convenient to use the latter. 

If $\tau \colon \X \to \QQQ/\PP^{\infty}$ is a fibration in $\Spc$,
then the homotopy pullback
\begin{equation*}
\X^\tau\equiv \X \times_{\QQQ/\PP^{\infty}} \QQQ \in \Mod_\X(\PP^{\infty})
\end{equation*}
is a $\PP^{\infty}$-module over $\X$,
in particular a $\PP^{\infty}$-module. 
Suppose $\tau\colon \X \to (\QQQ/\PP^{\infty})^{\fff}$ is a map in $\Spc$ with target some fibrant replacement of $\QQQ/\PP^{\infty}$. 
The model structure ensures there exists a functorial fibrant replacement $(\QQQ)^{\fff}$ of $\QQQ$ in $\Mod_{(\QQQ/\PP^{\infty})^{\fff}}(\PP^{\infty})$.
Using these fibrant replacements we define $\X^\tau$ by the homotopy pullback
\begin{equation*}
\X^\tau\equiv \X \times_{(\QQQ/\PP^{\infty})^{\fff}} (\QQQ)^{\fff} \in \Mod_\X(\PP^{\infty}).
\end{equation*}
Working in motivic symmetric spectra we note that $\Sigma^{\infty}\X^{\tau}_{+}$ is a strict $\Sigma^{\infty}\PP^{\infty}_{+}$-module.
(Recall that $\Sigma^{\infty}$ is a left Quillen functor for the injective motivic model structure \cite{Jardine}.)

The motivic twisted homological K-theory of $\tau\colon \X \to (\QQQ/\PP^{\infty})^{\fff}$ in $\Spc$ is defined as the derived smash product
\begin{equation*}
\KGL^{\tau}
\equiv
\Sigma^{\infty}\X^{\tau}_{+}
\wedge_{\Sigma^{\infty}\PP^\infty_{+}}
\KGL.
\end{equation*}
For the strict module structure on $\KGL$ we use the Bott inverted model discussed in \cite{Rondigs-Spitzweck-Ostvar}.
Likewise,
by running the same fibrant replacements, 
the motivic twisted cohomological K-theory of $\tau\colon \X \to (\QQQ/\PP^{\infty})^{\fff}$ is defined as the derived internal hom
\begin{equation*}
\KGL_{\tau}
\equiv
\uHom_{\Sigma^\infty \PP^\infty_+}(\Sigma^\infty \X^\tau_+,\KGL).
\end{equation*}

In the following we let $\B\BG$ denote the homotopy quotient $(\QQQ/\PP^{\infty})^{\fff}$.

\begin{proposition} \label{proposition:tauzero}
Suppose $S$ is a regular scheme and $\tau\colon\X\rto\B\BG$.
If $\tau$ is null then $\KGL^{\tau}$ is isomorphic to $\Sigma^{\infty}\X_{+}\wedge\KGL$ in the motivic stable homotopy category.
\end{proposition}
\begin{proof}
Corollary \ref{corollary:homotopypullback} identifies $\X^{\tau}_{+}$ with the smash product of motivic pointed spaces $\X_{+}\wedge\PP^{\infty}_{+}$. 
This follows because of the assumption on $\tau$ the former is obtained by first taking the homotopy pullback of the diagram
\begin{equation*}
\ast\rto\B\BG\lto\ast
\end{equation*}
and second forming the homotopy pullback along the canonical map $\X\rto\ast$.
Using this we obtain the isomorphisms
\begin{equation*}
\Sigma^{\infty}\X^{\tau}_{+}
\wedge_{\Sigma^{\infty}\PP^{\infty}_{+}}
\KGL
\cong
(\Sigma^{\infty}\X_{+}
\wedge
\Sigma^{\infty}\PP^{\infty}_{+})
\wedge_{\Sigma^{\infty}\PP^{\infty}_{+}}
\KGL
\cong
\Sigma^{\infty}\X_{+}\wedge\KGL.
\end{equation*}
The regularity assumption on the base scheme $S$ enters the proof of Corollary \ref{corollary:homotopypullback}, 
which we discuss next.
\end{proof}

Recall from \cite[\S2]{nmp-nonregular} the definition of the simplicial Picard functor $\nu\Pic$ on $\Sm$:
For a scheme $X$ in $\Sm$, 
let $\underline{\Pic}(X)$ denote the associated Picard groupoid.
Then the pseudo-functor $X\mapsto\underline{\Pic}(X)$ can be strictified to a presheaf on $\Sm$.
Applying the nerve functor to any such strictification defines the simplicial presheaf $\nu\Pic$ on $\Sm$.
\begin{lemma}
Suppose $S$ is a normal scheme.
Then the simplicial Picard functor $\nu\Pic$ is a Nisnevich local $\A^{1}$-invariant simplicial presheaf.
\end{lemma}
\begin{proof}
Nisnevich localness holds because the groupoid valued Picard functor $\Pic$ satisfies flat descent.
And $\A^{1}$-invariance holds because $\Pic$ is $\A^{1}$-invariant by assumption.
\end{proof}

We remark that $\nu\Pic$ is a commutative monoid in $\Spc$ by strictification.

\begin{lemma}
\label{lemma:sectionwisevpic}
Suppose $S$ is a regular scheme.
Applying the classifying space functor sectionwise to $\nu\Pic$ determines a Nisnevich local $\A^{1}$-invariant simplicial presheaf.
\end{lemma}
\begin{proof}
Since $S$ is regular, 
it is well known that the cohomology group $H^{2}_{\Nis}(X,\G)$ is trivial for every $X$ in $\Sm$. 
For an outline of a proof, 
we note that the sheaf $\MM_X^*$ of meromorphic functions on $X$ and the sheaf $\ZZ_X^1$ of codimension $1$ cycles on $X$ are flasque 
in the Nisnevich topology.
Thus,
using \cite[\S 21.6]{EGAIV}, 
the vanishing of $H^{2}_{\Nis}(X,\G)$ follows from the exact sequence 
\begin{equation}
\label{equation:divisorexactsequence}
0 \to \caO_X^* \to \MM_X^* \to \ZZ_X^1 \to 0.
\end{equation}
Let $B^{s}\nu\Pic$ be the sectionwise classifying space of $\nu\Pic$ and $\varphi \colon B^{s}\nu\Pic \to RB^{s}\nu\Pic$ a Nisnevich local replacement. 
Then $\pi_i((RB^{s}\nu\Pic)(X))=H^{2-i}_{\Nis}(X,\G)$ for $0 \le i \le 2$.
It follows that $\varphi$ is a sectionwise equivalence.
Thus the sectionwise classifying space construction is Nisnevich local (sectionwise equivalent to every Nisnevich local fibrant replacement) and $\A^{1}$-invariance is preserved.
\end{proof}
\begin{corollary}
\label{corollary:homotopypullback}
If $S$ is a regular scheme then the homotopy pullback of the diagram
\begin{equation*}
\ast\rto\B\BG\lto\ast
\end{equation*}
is isomorphic to $\PP^{\infty}$ in the motivic homotopy category.
\end{corollary}
\begin{proof}
Let $\HH^s(S)$ denote the homotopy category of simplicial presheaves on $\Sm$ with the objectwise model structure, 
and let $\HH(S)$ denote the motivic homotopy category.
Then the inclusion $\HH(S) \to \HH^s(S)$ arises from a right Quillen functor.
Therefore, 
in order to compute the homotopy pullback of $\ast\rto\B\BG\lto\ast$ in the motivic homotopy category, 
it is sufficient to compute the homotopy pullback of its image in $\HH^s(S)$. 

By Lemma \ref{lemma:sectionwisevpic}, 
$B^{s}\nu\Pic$ is an $\A^1$- and Nisnevich local replacement of $\B\BG$ (for notation see the proof of Lemma \ref{lemma:sectionwisevpic}).
The homotopy pullback of $\ast\rto B^{s}\nu\Pic\lto\ast$ in $\HH^s(S)$ is clearly $\BG$.
\end{proof}

\begin{remark}
The previous corollary would have been trivially true provided the infinity category of motivic spaces had been an infinity topos in the sense of \cite{Lurie}.
Alas, this is not true for motivic spaces:
Recall that in any infinity topos the loop functor provides an equivalence between the connected $1$-truncated pointed objects and discrete group objects. 
In motivic homotopy theory, 
the simplicial loop space of a $1$-truncated pointed motivic space is strongly $\A^{1}$-invariant.
Recall also that $\A^{1}$-invariant and strongly $\A^{1}$-invariant sheaves of groups are different notions.
A discrete motivic group object is synonymous with an $\A^{1}$-invariant sheaf of groups.
This shows that the equivalence does not hold in the motivic setting.
We thank J.~Lurie for this remark. 
It is unclear if Corollary \ref{corollary:homotopypullback} extends to an interesting class of base schemes.
We note that the sequence (\ref{equation:divisorexactsequence}) is exact if and only if $X$ is a locally factorial scheme.
However, 
a smooth scheme of finite type over a locally factorial scheme need not be locally factorial in general.
Thus we cannot expect that the group $H^{2}_{\Nis}(X,\G)$ is trivial over every locally factorial base scheme.
We thank M.~Levine for clarifying this remark.
\end{remark}

\begin{lemma}
\label{lemma:P00topoint}
There is an isomorphism of $\Sigma^\infty \PP^\infty_+$-modules $\Sigma^\infty \X^\tau_+ \wedge_{\Sigma^\infty \PP^\infty_+} \unit \cong \Sigma^\infty \X_+$
where $\Sigma^\infty \PP^\infty_+ \to \unit$ is induced by the canonical map $\PP^\infty \to \ast$.
\end{lemma}
\begin{proof}
For a fibration $\Y' \to \Y$ in $\Spc$ the pullback functor $\Spc/\Y \to \Spc/\Y'$ is a left Quillen functor for the injective motivic model structure on $\Spc$.
Indeed, 
it has a right adjoint and it preserves monomorphisms and weak equivalences. 
(Recall that $\Spc$ is right proper.)
Thus we obtain a left Quillen functor
$$\Mod_\Y(\PP^\infty) \to \Mod_{\Y'}(\PP^\infty).$$
This functor commutes with pushforward along the canonical map $\PP^\infty \to \ast$ and thus it preserves homotopy quotients by $\PP^\infty$. 
We deduce that the natural map $\QQQ\X^\tau/\PP^\infty \to \X$ is a weak equivalence, 
where $\QQQ\X^\tau \to \X^\tau$ is a cofibrant replacement in $\Mod_{\X}(\PP^\infty)$.
On the other hand, 
the forgetful functor $\Mod_{\X}(\PP^\infty) \to \Mod(\PP^\infty)$ is also a left Quillen functor.
Combining the above findings shows there is an isomorphism
$$\X^\tau \times^\bL_{\PP^\infty} \ast \cong \X.$$
Applying the motivic symmetric suspension spectrum functor yields the result.
\end{proof}

Next we consider some basic functorial properties of motivic twisted K-theory.
First we note there exits a functor
\begin{equation*}
\KGL^{-}
\colon
\Ho(\Spc/\B\BG)\rto\SH(S).
\end{equation*}
Note here that for a map from $\tau\colon\X\rto\B\BG$ to $\tau^{\prime}\colon\X^{\prime}\rto\B\BG$ there is an induced map between pullbacks 
$\X^{\tau}\rto (\X^{\prime})^{\tau^{\prime}}$ of $\PP^{\infty}$-modules, 
which induces a map of motivic symmetric spectra $\KGL^{\tau}\rto\KGL^{\tau^{\prime}}$.
Clearly this factors through the homotopy category of motivic spaces over $\B\BG$.
In particular, 
if $\tau$ and $\tau^{\prime}$ are $\A^{1}$-homotopy equivalent maps, 
then $\KGL^{\tau}$ and $\KGL^{\tau^{\prime}}$ are isomorphic.
We also note that $\KGL^{-}$ can be enhanced to a functor from $\Ho(\Spc/\B\BG)$ taking values in the category of highly structured $\KGL$-modules.
Likewise,
in the cohomological setup,
there exists a functor
\begin{equation*}
\KGL_{-}
\colon
\Ho(\Spc/\B\BG)^{\text{op}}\rto\SH(S).
\end{equation*}

Some parts of our discussion of motivic twisted K-theory rely on the notion of a ``motivic $E_{\infty}$-ring spectrum.''
For every operad $\mathcal{O}$ in motivic symmetric spectra,
the category of  $\mathcal{O}$-algebras acquires a combinatorial model structure.
A recent account of this model structure has been written up by Hornbostel in \cite{Hornbostel}.
Here, motivic symmetric spectra are viewed in the stable flat positive model structure.
In particular, 
there exists model structures for the commutative motivic operad and the image of the linear isometry operad in topological spaces.
A motivic $E_{\infty}$-ring spectrum is an algebra over a $\Sigma$-cofibrant replacement of the commutative motivic operad.
Motivic $E_{\infty}$-ring spectra and strict commutative motivic ring spectra are related by a Quillen equivalence \cite{Hornbostel}.
For our purposes we may therefore use these two notions interchangeably.

\section{A K{\"u}nneth isomorphism for homological motivic twisted K-theory}
\label{section:proofsofthemainresults}

An explicit computation furnishes a natural base change isomorphism expressing the $\KGL$-homology of $\PP^{\infty}$ in terms of $\KGL_{\ast}$ and unitary topological K-theory
\begin{equation}
\label{equation:firstisomorphism}
\KGL_{\ast}(\PP^{\infty})
\cong
\KGL_{\ast} \otimes_{\KU_{\ast}}\KU_{\ast}(\CP^{\infty}).
\end{equation}
The multiplicative structure on $\KGL_{\ast}(\PP^{\infty})$ induced from the $H$-space structure on $\PP^{\infty}$ can be read off from this isomorphism by using the ring structure on 
$\KU_{\ast}(\CP^{\infty})$ and the coefficient ring.
We refer to the work of Ravenel and Wilson \cite{Ravenel-Wilson} for a description of the ring structure on $\KU_{\ast}(\CP^{\infty})$ in terms of the multiplicative formal group law.

An application of motivic Landweber exactness \cite[Proposition 9.1]{nmp-nonregular} shows there is a natural base change isomorphism 
\begin{equation}
\label{equation:secondisomorphism}
\KGL_{\ast}\KGL
\cong
\KGL_{\ast} \otimes_{\KU_{\ast}}\KU_{\ast}\KU.
\end{equation}
The multiplicative structure on $\KGL_{\ast}\KGL$ induced from the ring structure on $\KGL$ can be read off from this isomorphism by using the ring structure on $\KU_{\ast}\KU$ 
and the coefficient ring.
For a description of the Hopf algebra $\KU_{\ast}\KU$ we refer to the work of Adams and Harris \cite[Part II, \S13]{Adams}.
\begin{lemma}
\label{lemma:betamapsto}
Under the naturally induced composite map
\begin{equation*}
\KGL_{\ast}(\PP^{\infty})
\rto
\KGL_{\ast}\KGL
\rto
\KGL_{\ast}
\end{equation*}
the generator $\beta_{i}$ maps to $1$ if $i=0,1$ and to $0$ if $i\neq 0,1$.
\end{lemma}
\begin{proof}
We note that $\KGL_{\ast}(\PP^{\infty})$ is a free $\KGL_{\ast}$-module generated by the elements $\beta_{i}$ for $i\geq 0$.
Thus the claim follows from the analogous result for unitary topological K-theory of $\CP^{\infty}$,
see \cite{Khorami} for example, 
by applying the functor $\KGL_{\ast} \otimes_{\KU_{\ast}}-$ and appealing to the base change isomorphisms (\ref{equation:firstisomorphism}) and (\ref{equation:secondisomorphism}).
\end{proof}

Lemma \ref{lemma:betamapsto} verifies the previous assertion that $\KGL_{\ast}$ is not a flat $\KGL_{\ast}(\PP^{\infty})$-module.
Next we establish a result which is pivotal for our proof of the vanishing of the Tor-groups discussed in Section \ref{section:mainresults}.
\begin{lemma}
\label{lemma:flatnesslocalization}
The naturally induced map $\KGL_{\ast}(\PP^{\infty})\rto\KGL_{\ast}\KGL$ is a flat ring map.
\end{lemma}
\begin{proof}
In \cite{Gepner-Snaith} and \cite{Spitzweck-Ostvar} it is shown that $\KGL$ is isomorphic to the Bott inverted motivic suspension spectrum of $\PP^{\infty}_{+}$.
Thus the map in question is a localization.
In particular it is flat.
For an alternate proof, 
combine the base change isomorphisms (\ref{equation:firstisomorphism}) and (\ref{equation:secondisomorphism}) with flatness of the naturally induced map 
$\KU_{\ast}(\CP^{\infty})\rto\KU_{\ast}\KU$.
(This map is a localization according to the topological analogue of our first argument,
which is well known and follows from the motivic result by taking complex points, 
or alternatively by \cite{Khorami}.)
\end{proof}

In Proposition \ref{proposition:Hopfalgebroidstructure} we show that $(\KGL_{\ast}(\PP^{\infty}),\KGL_{\ast}\KGL\otimes_{\KGL_{\ast}}\KGL_{\ast}(\PP^{\infty}))$ has the structure of a 
flat graded Adams Hopf algebroid.
We refer the reader to Section \ref{section:hopfalgebroids} for background on the notions and results appearing in the formulation and proof of the following key result.

\begin{theorem}
\label{theorem:landweberexactness}
The naturally induced composite map
\begin{equation*}
\KGL_{\ast}(\PP^{\infty})
\rto
\KGL_{\ast}\KGL
\rto
\KGL_{\ast}
\end{equation*}
is Landweber exact for the flat graded Adams Hopf algebroid 
\begin{equation*}
(\KGL_{\ast}(\PP^{\infty}),\KGL_{\ast}\KGL\otimes_{\KGL_{\ast}}\KGL_{\ast}(\PP^{\infty})).
\end{equation*}
\end{theorem}
\begin{proof}
By Lemma \ref{lemma:landweberexctnessflatness} it suffices to show that the left unit map 
\begin{eqnarray*}
\eta_{\KGL_{\ast}(\PP^{\infty})}
\colon
\KGL_{\ast}(\PP^{\infty})
& \rto &
\KGL_{\ast}\KGL\otimes_{\KGL_{\ast}}\KGL_{\ast}(\PP^{\infty}) \\
& \cong & (\KGL_{\ast}\KGL\otimes_{\KGL_{\ast}}\KGL_{\ast}(\PP^{\infty}))\otimes_{\KGL_{\ast}(\PP^{\infty})}\KGL_{\ast}(\PP^{\infty}) \\
& \rto & (\KGL_{\ast}\KGL\otimes_{\KGL_{\ast}}\KGL_{\ast}(\PP^{\infty}))\otimes_{\KGL_{\ast}(\PP^{\infty})}\KGL_{\ast}
\end{eqnarray*}
for the displayed Hopf algebroid is flat (the target is canonically isomorphic to $\KGL_{\ast}\KGL$).
Remark \ref{remark:hopfalgebroidremark} provides more details on the Hopf algebroid structure.

The left unit map  $\eta_{\KGL_{\ast}(\PP^{\infty})}$ and its topological analogue $\eta_{\KU_{\ast}(\CP^{\infty})}$ determines a commutative diagram where the horizontal maps are the 
base change isomorphisms given in (\ref{equation:firstisomorphism}) and (\ref{equation:secondisomorphism}):
\begin{equation*}
\xymatrix{
\KGL_{\ast}(\PP^{\infty})\ar[d]_{\eta_{\KGL_{\ast}(\PP^{\infty})}} \ar[r]^-{\cong}  & \KGL_{\ast}\otimes_{\KU_{\ast}}\KU_{\ast}(\CP^{\infty}) \ar[d] ^{\KGL_{\ast}\otimes\eta_{\KU_{\ast}(\CP^{\infty})}}\\
\KGL_{\ast}\KGL \ar[r]^-{\cong}  &
\KGL_{\ast}\otimes_{\KU_{\ast}}\KU_{\ast}\KU }
\end{equation*}
Khorami \cite{Khorami} has shown that $\eta_{\KU_{\ast}(\CP^{\infty})}$ coincides with the naturally induced map from $\KU_{\ast}(\CP^{\infty})$ to $\KU_{\ast}\KU$.
By motivic Landweber exactness \cite{nmp-nonregular} we deduce that $\eta_{\KGL_{\ast}(\PP^{\infty})}$ coincides with the naturally induced flat map in Lemma \ref{lemma:flatnesslocalization}.
This finishes the proof.
\end{proof}

By Landweber exactness the functor from comodules over $\KGL_{\ast} \KGL \otimes_{\KGL_{\ast}}\KGL_{\ast}(\PP^{\infty})$ to $\KGL_{\ast}$-algebras is exact
(where $\KGL_{\ast}$ is viewed with its $\KGL_{\ast}(\PP^{\infty})$-algebra structure).
This observation is a crux input in the proof of the next result.
\begin{corollary}
Suppose $E$ is a comodule over $\KGL_{\ast} \KGL \otimes_{\KGL_{\ast}}\KGL_{\ast}(\PP^\infty)$. 
The group 
\begin{equation*}
\Tor^{\KGL_{\ast}(\PP^{\infty})}_{\ast}(E,\KGL_{\ast})
\end{equation*}
is trivial in positive degrees.
\end{corollary}
\begin{proof}
Proposition \ref{proposition:Hopfalgebroidstructure} implies there exists a projective model structure on the category of non-connective chain complexes of 
$\KGL_{\ast} \KGL \otimes_{\KGL_{\ast}}\KGL_{\ast}(\PP^\infty)$-comodules for the set of dualizable comodules \cite[Theorem 2.1.3]{Hovey}.
The projective model structure is proper, finitely generated, stable symmetric monoidal and satisfies the monoid axiom.
Moreover, 
a map is a cofibration if and only if it is a degreewise split monomorphism with cofibrant cokernel.
The cofibrant objects are retracts of certain sequential cell-complexes described in details in \cite[Theorem 2.1.3]{Hovey}.
These results are easily transferred to the graded setting.

Due to the existence of the projective model structure we are entitled to a cofibrant replacement $\QQQ E\rto E$ (recall this is a projective weak equivalence and $\QQQ E$ is cofibrant).
Proposition \ref{proposition:Hopfalgebroidstructure} shows that $(\KGL_{\ast}(\PP^{\infty}),\KGL_{\ast}\KGL\otimes_{\KGL_{\ast}}\KGL_{\ast}(\PP^{\infty}))$ is a graded Adams Hopf algebroid.
This additional structure guarantees that every weak equivalence in the projective model structure is a quasi-isomorphism \cite[Proposition 3.3.1]{Hovey}.

We claim that the Tor-groups in question are computed by the homology of the chain complex 
\begin{equation*}
\QQQ E \otimes_{\KGL_{\ast}(\PP^{\infty})}\KGL_{\ast}.
\end{equation*}
The proof proceeds by comparing chain complexes of comodules over the tensor product $\KGL_{\ast} \KGL \otimes_{\KGL_{\ast}}\KGL_{\ast}(\PP^\infty)$ with chain complexes of 
$\KGL_{\ast}(\PP^{\infty})$-modules. 
Indeed, 
by \cite[Proposition 1.3.4]{Hovey}, 
cf.~the proof of \cite[Theorem 2.1.3]{Hovey}, 
the generating cofibrations are of such a form that $\QQQ E$ is even cofibrant as a complex of $\KGL_{\ast}(\PP^{\infty})$-modules for the usual projective model structure.
(We also note that the tensor factor $\KGL_{\ast}$ need not be cofibrantly replaced because the monoid axiom holds in the projective model structure.)

As noted earlier there exists an exact functor from the category of comodules over $\KGL_{\ast} \KGL \otimes_{\KGL_{\ast}}\KGL_{\ast}(\PP^{\infty})$ to $\KGL_{\ast}$-algebras.
We note that any such functor preserves quasi-isomorphisms.
In particular there is a quasi-isomorphism
\begin{equation*}
\QQQ E \otimes_{\KGL_{\ast}(\PP^{\infty})}\KGL_{\ast}
\simeq 
E \otimes_{\KGL_{\ast}(\PP^{\infty})}\KGL_{\ast}.
\end{equation*}
By combining the above we conclude that the Tor-groups vanish in positive degrees.
\end{proof}

By strong convergence of the spectral sequence (\ref{equation:torss}) we are almost ready to conclude the proof of the K{\"u}nneth isomorphism in Theorem \ref{theorem:edgemapisomorphism}.
It only remains to observe that $\KGL_{\ast}(\X^{\tau})$ is a comodule over the Hopf algebroid 
\begin{equation*}
(\KGL_{\ast}(\PP^{\infty}),\KGL_{\ast}\KGL\otimes_{\KGL_{\ast}}\KGL_{\ast}(\PP^{\infty})).
\end{equation*}
To begin with, 
the naturally induced action of $\PP^{\infty}$ on $\X^{\tau}$ yields a map
\begin{equation*}
\KGL_{\ast}(\PP^{\infty}\times\X^{\tau})
\rto
\KGL_{\ast}(\X^{\tau}).
\end{equation*}
Since $\KGL_{\ast}(\PP^{\infty})$ is free over the coefficient ring $\KGL_{\ast}$ there is an isomorphism
\begin{equation*}
\KGL_{\ast}(\PP^{\infty}\times\X^{\tau})
\cong
\KGL_{\ast}(\PP^{\infty})
\otimes_{\KGL_{\ast}}
\KGL_{\ast}(\X^{\tau}).
\end{equation*}
It follows that $\KGL_{\ast}(\X^{\tau})$ is a module over $\KGL_{\ast}(\PP^{\infty})$.
Using the unit map from the motivic sphere spectrum $\unit$ to $\KGL$ we get a map between motivic spectra
\begin{eqnarray*}
\KGL\wedge\Sigma^{\infty}\X^{\tau}_{+}
\cong
\KGL\wedge\unit\wedge\Sigma^{\infty}\X^{\tau}_{+}
\rto
\KGL\wedge\KGL\wedge\Sigma^{\infty}\X^{\tau}_{+}.
\end{eqnarray*}
From this we immediately obtain the desired comodule map
\begin{eqnarray*}
\KGL_{\ast}(\X^{\tau})
& \rto &
\KGL_{\ast}\KGL\otimes_{\KGL_{\ast}}\KGL_{\ast}(\X^{\tau}) \\
& \cong &
(\KGL_{\ast}\KGL\otimes_{\KGL_{\ast}}\KGL_{\ast}(\PP^{\infty}))\otimes_{\KGL_{\ast}(\PP^{\infty})}\KGL_{\ast}(\X^{\tau}).
\end{eqnarray*}
This is clearly a coassociative and unital map between $\KGL_{\ast}(\PP^{\infty})$-modules. 

\begin{remark}
\label{remark:BBGtrivial}
As noted earlier the $\KGL_{\ast}$-module $\KGL_{\ast}(\PP^{\infty})$ is free on the generators $\beta_{i}$ for $i\geq 0$.
Its multiplicative structure can be described in terms of power series.
Modulo the problem of computing the coefficient ring $\KGL_{\ast}$ this leaves us with investigating the K-theory of the homotopy fiber $\X^{\tau}$.
On the other hand, 
an inspection of the module structures in Theorem \ref{theorem:edgemapisomorphism},
cf.~Lemma \ref{lemma:betamapsto}, 
reveals there is an isomorphism
\begin{equation*}
\KGL_{\ast}^{\tau}(\X)
\cong
\KGL_{\ast}(\X^{\tau})/(\beta_{0}-1,\beta_{1}-1,\beta_{i})_{i\geq 2}.
\end{equation*}

In case $\tau$ is the identity map on $\B\BG$ then $\KGL_{\ast}(\B\BG^{\tau})\cong\KGL_{\ast}$.
We claim that $\KGL_{\ast}^{\tau}(\B\BG)$ is the trivial group. 
This follows by comparing the images of $\beta_{0}$ or $\beta_{1} $ in the respective tensor factors.
For example, 
the class $\beta_{1}$ maps to the unit in $\KGL_{\ast}$ and to zero in $\KGL_{\ast}(\B\BG^{\tau})$.
\end{remark}

Next we turn to the constructions of the twisted Chern characters.
The proof of our main result Theorem \ref{theorem:chernisomorphism} relies on results in \cite{Cisinski-Deglise} and \cite{nmp-exactness}.
Let $\MQ$ denote the motivic Eilenberg-MacLane spectrum introduced in \cite{Voevodsky-icm}.
(We refer to \cite{DRO} for a definition of $\MQ$ viewed as a motivic functor.) 
It has the structure of a commutative monoid in the category of motivic symmetric spectra \cite{Rondigs-Ostvar1}, \cite{Rondigs-Ostvar2}.
The periodization $\PM\Q$ of $\MQ$ is also highly structured:
Form the free commutative $\MQ$-algebra $\PM\Q_{\geq 0}$ on one generator in degree $(2,1)$ 
(perform this in $\PP^{1}$-spectra of simplicial presheaves of $\Q$-vector spaces and then transfer the spectrum back to obtain a strictly commutative ring spectrum).
Inverting the same generator following the method in \cite{Rondigs-Spitzweck-Ostvar} produces the commutative monoid $\PM\Q$ whose underlying spectrum is the
infinite wedge sum $\bigvee_{i \in \Z} \Sigma^{2i,i} \MQ$.
\begin{lemma}
\label{lemma:PMQisKGLMQ}
There is an $E_{\infty}$-isomorphism between $\PM\Q$ and $\KGL\wedge\MQ$.
\end{lemma}
\begin{proof}
By the universal property of $\PM\Q_{\geq 0}$ there is a commutative diagram:
\begin{equation*}
\xymatrix{
& \Sigma^{\infty}\PP^{\infty}_{+}\wedge\MQ\ar[d]  \\
\PM\Q_{\geq 0}\ar[ur]\ar[r] & \KGL\wedge\MQ }
\end{equation*}
The generator in degree $(2,1)$ maps to the canonical element in $\Sigma^{\infty}\PP^{\infty}_{+}\wedge\MQ$ determined by the Bott element of 
$\Sigma^{\infty}\PP^{\infty}_{+}$ \cite{Spitzweck-Ostvar}.
The diagonal map is an isomorphism.
Inverting the generator and the Bott element gives the desired isomorphism.
\end{proof}

Lemma \ref{lemma:PMQisKGLMQ} furnishes an $\Sigma^{\infty}\PP^{\infty}_{+}$-algebra structure on $\PM\Q$ via the
map $$\Sigma^{\infty}\PP^{\infty}_{+} \to \KGL \to \KGL \wedge \MQ \cong \PM\Q.$$

Combining Lemma \ref{lemma:PMQisKGLMQ} and the canonical ring map $\KGL\rto\KGL\wedge\MQ$ we arrive at the Chern character
\begin{equation}
\label{equation:cherncharacter}
\Ch
\colon
\KGL
\rto
\PM\Q
\end{equation}
from algebraic K-theory to the periodized rational motivic Eilenberg-MacLane spectrum.
(This is a map of motivic ring spectra.)
For any twist $\tau$, 
smashing (\ref{equation:cherncharacter}) with $\X^{\tau}$ in the homotopy category of $\Sigma^{\infty}\PP^{\infty}_{+}$-modules defines the twisted Chern character
\begin{equation}
\label{equation:twistedcherncharacter}
\Ch^{\tau}
\colon
\KGL^{\tau}
\rto
\PM^{\tau}\Q.
\end{equation}
As asserted in Theorem \ref{theorem:chernisomorphism}, 
the rationalization of (\ref{equation:twistedcherncharacter}) is an isomorphism for geometrically unibranched excellent base schemes.
This follows by combining \cite[Corollary 15.1.6]{Cisinski-Deglise} and \cite[Theorem 10.1, Corollary 10.3]{nmp-exactness}.
 
Similarly we define the cohomological Chern character
\begin{equation}
\label{equation:twistedcherncharacter2}
\Ch_{\tau}
\colon
\KGL_{\tau}
\rto
\PM_{\tau}\Q
\end{equation}
by taking internal hom-objects from $\Sigma^\infty \X^\tau_+$ into the untwisted Chern character $\Ch$. 
We note that the rationalization of (\ref{equation:twistedcherncharacter2}) is an isomorphism over geometrically unibranched excellent base schemes provided $\Sigma^\infty \X^\tau_+$ 
is strongly dualizable in $\Ho (\Sigma^\infty \PP^\infty_+ - \Mod)$.
Indeed, 
this follows immediately by smashing the rational isomorphism in (\ref{equation:cherncharacter}) with the dual of $\Sigma^\infty \X^\tau_+$.

\begin{remark}
In the topological setup, 
Atiyah and Segal \cite{Atiyah-Segal2} employed a different method in order to construct a Chern character for twisted K-theory and a corresponding theory of Chern classes.
We leave the comparison of the two constructions as an open question.
\end{remark}

We end this section by outlining computations of nontrivial twisted K-groups for the motivic $(3,1)$-sphere.
To begin with we allow the base scheme to be an arbitrary field.
In the interest of explicit computations in all degrees, 
we specialize to finite fields.

Recall the smash product decomposition $S^{3,1}=S^{2}\wedge\G$ for the motivic $(3,1)$-sphere.
Moreover, 
there is a homotopy pushout square of motivic spaces:
\begin{equation*}
\xymatrix{
\PP^{1} \ar[r] \ar[d] & \ast \ar[d] \\
\ast\ar[r] & S^{3,1} }
\end{equation*}
We shall consider the twist $\tau_{n}\colon S^{3,1} \rto\B\BG$ corresponding to $n$ times the canonical map $S^{3,1} \rto\B\BG$.
Precomposing with the map $\ast\rto S^{3,1}$ produce null homotopic twists on $\PP^{1}$ and the point.
In order to proceed we infer,
leaving details to the interested reader, 
there is a homotopy pushout diagram:
\begin{equation*}
\xymatrix{
\PP^{\infty}\times\PP^{1} \ar[r] \ar[d] & \PP^{\infty}\times (\PP^{\infty})^{n} \ar[r] & \PP^{\infty}\ar[d] \\
\PP^{\infty} \ar[rr] && (S^{3,1})^{\tau_{n}}   } 
\end{equation*}
The left vertical map is the projection on the first factor.
The upper composite horizontal map arise from embedding $\PP^{1}$ into $(\PP^{\infty})^{n}$ along the diagonal map $\PP^{\infty}\subseteq (\PP^{\infty})^{n}$ and using the $H$-space 
structure on the infinite projective space.
With this in hand we get an induced long exact sequence
\begin{equation}
\label{equation:taunles}
\cdots
\rto
\Sigma^{2,1}\KGL_{\ast}\oplus\KGL_{\ast}
\rto
\KGL_{\ast}\oplus\KGL_{\ast}
\rto
\KGL^{\tau_{n}}_{\ast}(S^{3,1})
\rto
\cdots.
\end{equation}
Next we infer that the map between the direct sums in (\ref{equation:taunles}) is uniformly given by 
\begin{equation}
\label{equation:taunformula}
(a,b)
\mapsto
(an\beta+b,-b).
\end{equation}
Again we leave the details to the interested reader.
(Note that (\ref{equation:taunformula}) is compatible with its evident topological counterpart.)
From (\ref{equation:taunles}) we deduce the exact sequence
\begin{equation}
\label{equation:taunend}
K_{1} \oplus K_{1}\to 
K_{1} \oplus K_{1} \to
K^{\tau_{n}}_{1} (S^{3,1})\to 
K_{0} \oplus K_{0}\to 
K_{0} \oplus K_{0} \to 
K^{\tau_{n}}_{0}(S^{3,1}) \to 0.
\end{equation}
Using (\ref{equation:taunend}) and the fact that $K_{0}$ is infinite cyclic for any field we read off the isomorphism
\begin{equation*}
K^{\tau_{n}}_{0}(S^{3,1})
\cong
\Z/n,
\end{equation*}
where, in general, $K^{\tau}_i(\X)$ is shorthand for $\KGL^\tau_{i,0}(\X)$, $i\in\Z$.
By specializing to a finite field $\mathbb{F}_q$ and an odd integer $i\geq 1$, 
we deduce the exact sequence
\begin{equation}
\label{equation:taunses}
0 \to 
K^{\tau_{n}}_{i+1}(S^{3,1}_{\mathbb{F}_q}) \to 
K_{i} \oplus K_{i} \to 
K_{i} \oplus K_{i}  \to 
K^{\tau_{n}}_{i}(S^{3,1}_{\mathbb{F}_q}) \to 0.
\end{equation}
This follows from (\ref{equation:taunles}) since the $K$-groups for finite fields vanish in positive even degrees \cite{Quillen}.
Combining (\ref{equation:taunformula}) and (\ref{equation:taunses}) yields the isomorphisms
\begin{equation*}
K_{2i}^{\tau_n}(S^{3,1}_{\mathbb{F}_q})
\cong
\ker(\Z/(q^i-1) \overset{\times n}{\longrightarrow}\Z/(q^i-1))
\cong 
\Z/\mathrm{gcd}(n,q^i-1)
\end{equation*}
and
\begin{equation*}
K_{2i-1}^{\tau_n}(S^{3,1}_{\mathbb{F}_q})
\cong
\Z/\mathrm{gcd}(n,q^i-1).
\end{equation*}

\section{Spectral sequences for motivic twisted K-theory}
\label{section:spectralsequencesformotivictwisted K-theory}
In this section we shall construct and show strong convergence of the spectral sequences relating motivic (co)homology to motivic twisted K-theory.
The review of this material in Section \ref{section:mainresults} provides motivation and background from K-theory.
Our approach employs the slice tower formalism introduced by Voevodsky \cite{Voevodsky-slices} and further developed from the viewpoint of colored operads in \cite{GSO}.

Let $r_{i}$ denote the right adjoint of the natural inclusion functor $\Sigma_T^i \SH(S)^\eff \subseteq \SH(S)$.
Define $f_i \colon \SH(S) \to \SH(S)$ as the composite functor
\begin{equation*}
\SH(S) \overset{r_{i}}{\to} \Sigma_T^i \SH(S)^\eff \subseteq \SH(S).
\end{equation*}
In \cite{Voevodsky-slices} the $i$th slice $s_i(X)$ of $X$ is defined as the cofiber of the canonical map 
$$f_{i+1}(X) \to f_i (X).$$ 
In the companion paper \cite{GSO} we show that $f_0$ and $s_0$ respect motivic $E_\infty$-structures, 
and $f_q$ and $s_q$ respect module structures over $E_\infty$-algebras.
Recall that $f_{0}$ is reminiscent of the connective cover in topology.
As a sample result we state the following key result.
\begin{theorem} 
\label{theorem:f0einfty}
Suppose $A$ is an $A_\infty$- or an $E_\infty$-algebra in $\Spt_T^\Sigma(S)$.
Then $f_0(A)$ is naturally equipped with the structure of an $A_\infty$- resp.~$E_\infty$-algebra.
The canonical map $f_0 A \to A$ can be modelled as a map of $A_\infty$- resp.~$E_\infty$-algebras. 
\end{theorem}
The corresponding statements dealing with $s_0$ and modules are formulated in \cite{GSO}.
In the interest of keeping this paper concise we refer to loc.~cit.~for further details.

We define the connective K-theory spectrum $\kgl$ to be $f_0\KGL$. 
With this definition,
$\kgl$ is a $\Sigma^\infty \PP^\infty_+$-module because the $E_\infty$-map $\Sigma^\infty \PP^\infty_+ \to\KGL$ factors uniquely through the connective K-theory spectrum.
Here we use that $f_0$ is a lax monoidal functor that respects $E_\infty$-objects.
More generally, 
$f_i\KGL=\Sigma^{2i,i}\kgl$ is a $\Sigma^\infty\PP^\infty_+$-module.
(The two possible module structures, using either the shift functor or the fact that $f_i$ produces a module over $f_0$, coincide.)
Moreover, 
$f_{i+1} \KGL \to f_i \KGL$ is a $\kgl$-module map, 
hence a $\Sigma^\infty\PP^\infty_+$-module map.
By stitching these maps together we obtain a sequential filtration of $\KGL$ by shifted copies of the connective K-theory spectrum
\begin{equation}
\label{equation:KGLfiltration}
\cdots 
\rto
\Sigma^{2i+2,i+1} \kgl 
\to 
\Sigma^{2i,i}\kgl 
\to 
\cdots 
\to 
\KGL.
\end{equation}
The maps in (\ref{equation:KGLfiltration}) are $\Sigma^\infty\PP^\infty_+$-module maps.
Hence for every twist $\tau\colon\X\rto\B\B\G$ there is an induced filtration of the motivic twisted K-theory spectrum
\begin{equation}
\label{equation:tower1}
\cdots
\rto
\Sigma^\infty \X^\tau_{+} \wedge_{\Sigma^\infty \PP^\infty_+}\Sigma^{2i+2,i+1}\kgl
\rto
\Sigma^\infty \X^\tau_{+} \wedge_{\Sigma^\infty \PP^\infty_+}\Sigma^{2i,i}\kgl
\rto
\cdots
\end{equation}
\begin{equation*}
\rto
\cdots
\rto
\Sigma^\infty \X^\tau_{+} \wedge_{\Sigma^\infty \PP^\infty_+}\KGL=
\KGL^{\tau}.
\end{equation*}
Likewise, 
by applying the functor $\uHom_{\Sigma^\infty \PP^\infty_+}(\Sigma^\infty \X^\tau_+, -)$ to the filtration (\ref{equation:KGLfiltration}) of $\KGL$ we obtain a filtration of $\KGL_\tau$ taking the form
\begin{equation}
\label{equation:tower2}
\cdots
\rto
\uHom_{\Sigma^\infty \PP^\infty_+}(\Sigma^\infty \X^\tau_+,\Sigma^{2i+2,i+1}\kgl)
\rto
\uHom_{\Sigma^\infty \PP^\infty_+}(\Sigma^\infty \X^\tau_+,\Sigma^{2i,i}\kgl)
\rto
\cdots
\end{equation}
\begin{equation*}
\rto
\cdots
\rto
\uHom_{\Sigma^\infty \PP^\infty_+}(\Sigma^\infty \X^\tau_+,\KGL)=
\KGL_{\tau}.
\end{equation*}

Our next objective is to identify the filtration quotients $\QQ_{i}( \X^\tau)$ of the tower (\ref{equation:tower1}) and $\QQ^i(\X^\tau)$ of the tower (\ref{equation:tower2}).
Note that the tower (\ref{equation:tower1}) gives rise to an exact couple by applying homotopy groups for a fixed weight:
\begin{equation*}
\xymatrix{
&
\pi_{\ast}\Sigma^\infty \X^\tau_{+} \wedge_{\Sigma^\infty \PP^\infty_+}\Sigma^{2i+2,i+1}\kgl
\ar[rr] 
&&
\pi_{\ast}\Sigma^\infty \X^\tau_{+} \wedge_{\Sigma^\infty \PP^\infty_+}\Sigma^{2i,i}\kgl
\ar[dl]
\\
&&
\pi_{\ast}\QQ_{i}( \X^\tau) \ar[ul]       }
\end{equation*}
Similarly, 
the tower (\ref{equation:tower2}) gives rise to an exact couple featuring the quotients $\QQ^i(\X^\tau)$. 
Following a standard process we obtain spectral sequences with target graded groups $\KGL^{\tau}_*$ and $\KGL_\tau^*$.
In the following we analyze these spectral sequences in details when the base scheme is a perfect field.

From now on we assume that the base scheme $S$ is a perfect field.
Using the slice computations of $\KGL$ in \cite{Levine-slices} and \cite{Voevodskymssktheory}, \cite{Voevodsky-zeroslices}, 
there is an exact triangle of $\Sigma^\infty \PP^\infty_+$-modules
\begin{equation*}
\Sigma^{2,1}\kgl \to \kgl \to \MZ \to \Sigma^{3,1}\kgl.
\end{equation*}
Thus the filtration quotient $\QQ_i(\X^\tau)$ is isomorphic to 
\begin{equation*}
\Sigma^\infty \X^\tau_{+} \wedge_{\Sigma^\infty \PP^\infty_+}\Sigma^{2i,i}\MZ,
\end{equation*}
whereas the filtration quotient $\QQ^i(\X^\tau)$ is isomorphic to
\begin{equation*}
\uHom_{\Sigma^\infty \PP^\infty_+}(\Sigma^\infty \X^\tau_+,\MZ).
\end{equation*}

\begin{lemma}
\label{lemma:s0ofP00}
The unit map $\unit \rto \Sigma^\infty \PP^\infty_+ $ induces an isomorphism on zero slices.
\end{lemma}
\begin{proof}
Induction on the cofiber sequence
\begin{equation*}
\Sigma^\infty \PP^{n-1}_+ 
\rto
\Sigma^\infty \PP^n_+ 
\rto
\Sigma^{2n,n} \unit
\end{equation*}
gives the isomorphism
\begin{equation*}
s_{0} \unit
\cong
s_{0}\Sigma^\infty \PP^n_+.
\end{equation*}
To conclude we use that $s_{0}$ commutes with homotopy colimits,
cf.~\cite[Lemma 4.4]{Spitzweckslices}.
\end{proof}

\begin{lemma}
\label{lemma:E00factorization}
The diagram of $E_\infty$-ring spectra
$$
\xymatrix{\Sigma^\infty \PP^\infty_+ \ar[r] \ar[rd] & \kgl \ar[r] & \MZ \\
& \unit \ar[ru] &}
$$
commutes.
\end{lemma}
\begin{proof}
By Lemma \ref{lemma:s0ofP00} and the isomorphism $s_{0}\unit\cong\MZ$, 
applying the zero slice functor to the maps $\Sigma^\infty \PP^\infty_+ \to \kgl$ and $\Sigma^\infty \PP^\infty_+ \to \unit$ produces diagrams of $E_\infty$-ring spectra:
$$
\xymatrix{\Sigma^\infty \PP^\infty_+ \ar[r] \ar[d] & \kgl \ar[d] \\
\MZ \ar[r] & \MZ} \hspace{1cm}
\xymatrix{\Sigma^\infty \PP^\infty_+ \ar[r] \ar[d] & \unit \ar[d] \\
\MZ \ar[r] & \MZ}
$$
Now, 
since the construction of $E_\infty$-structures on zero slices in \cite{GSO} is not transparently functorial, 
a trick is required in order to verify commutativity of the two diagrams.
This follows by applying the localization machinery of \cite{GSO} to the two-colored operad whose algebras comprise maps between $E_\infty$-algebras.
\end{proof}

\begin{theorem}
\label{theorem:filtrationquotient}
There exists an isomorphism in the motivic stable homotopy category between the filtration quotient $\QQ_{i}( \X^\tau)$ of (\ref{equation:tower1}) and the $(2i,i)$-suspension 
of the motive $\MZ\wedge\Sigma^{\infty}\X_{+}$ of $\X$.
Likewise, 
there exists an isomorphism between the filtration quotient $\QQ^i(\X^\tau)$ of (\ref{equation:tower2}) and the $(2i,i)$-suspension of the internal hom $\uHom(\Sigma^\infty \X_+,\MZ)$.
\end{theorem}
\begin{proof}
It suffices to consider the case $i=0$.
The $0$th filtration quotient $\QQ_{0}( \X^\tau)$ identifies with $\Sigma^\infty \X^\tau_{+} \wedge_{\Sigma^\infty \PP^\infty_+}\MZ$.
By Lemma \ref{lemma:E00factorization} there is an isomorphism
\begin{equation*}
\Sigma^\infty \X^\tau_{+} \wedge_{\Sigma^\infty \PP^\infty_+}\MZ
\cong
(\Sigma^\infty \X^\tau_{+} \wedge_{\Sigma^\infty \PP^\infty_+}\unit)\wedge_{\unit}\MZ.
\end{equation*}
Lemma \ref{lemma:P00topoint} implies there is an isomorphism
\begin{equation*}
(\Sigma^\infty \X^\tau_{+} \wedge_{\Sigma^\infty \PP^\infty_+}\unit)\wedge_{\unit}\MZ
\cong
\Sigma^\infty \X_{+} \wedge\MZ.
\end{equation*}
The proof of the statement for $\QQ^i(\X^\tau)$ proceeds similarly by comparing the module categories over $\Sigma^\infty \PP^\infty_+$ and $\MZ$ via the isomorphisms 
\begin{align*}
\QQ^{0}(\X^\tau) & \cong 
\uHom_{\Sigma^\infty \PP^\infty_+}(\Sigma^\infty \X^\tau_+,\MZ)\cong \uHom_{\MZ}(\Sigma^\infty \X^\tau_+ \wedge_{\Sigma^\infty \PP^\infty_+} \MZ,\MZ) \\
& \cong 
\uHom_{\MZ}(\Sigma^\infty \X_+ \wedge \MZ,\MZ) \cong \uHom(\Sigma^\infty \X_+,\MZ).
\end{align*}
\end{proof}

The isomorphisms in Theorem \ref{theorem:filtrationquotient} are clearly functorial in $\X$ and $\tau$. 
It is important to note that the filtration quotients $\QQ_{i}( \X^\tau)$ and $\QQ^i(\X^\tau)$ are independent of the twist.

Theorem \ref{theorem:filtrationquotient} implies there exist spectral sequences 
\begin{equation}
\label{equation:twistedKss1}
\MZ_{\ast}(\Sigma^\infty \X_{+} )
\Longrightarrow
\KGL^{\tau}_{\ast}(\X)
\end{equation}
and
\begin{equation}
\label{equation:twistedKss2}
\MZ^{\ast}(\Sigma^\infty \X_{+} )
\Longrightarrow
\KGL_{\tau}^{\ast}(\X)
\end{equation}
relating motivic homology and cohomology to motivic twisted K-theory. 
In what follows we shall discuss the convergence properties of (\ref{equation:twistedKss1}) and (\ref{equation:twistedKss2}).
Our approach makes use of the notion of ``very effectiveness'' which is of independent interest in motivic homotopy theory over 
any base scheme $S$.
In order to make this precise we introduce the following subcategory of $\SH(S)$.
\begin{definition}
The very effective motivic stable homotopy category $\SH(S)^{\Veff}$ is the smallest full subcategory of $\SH(S)$ that contains all 
suspension spectra of smooth schemes of finite type over $S$ and is closed under extensions and homotopy colimits.
\end{definition}

We note that $\SH(S)^{\Veff}$ is not a triangulated category since it is not closed under simplicial desuspension.
However,
it is a subcategory of the effective motivic stable homotopy category,
which we denote by $\SH(S)^{\eff}$.
Finally, 
we remark that $\SH(S)^{\Veff}$ forms the homologically positive part of $t$-structures on $\SH(S)$ and $\SH(S)^{\eff}$.

\begin{lemma}
\label{lemma:veffclosedtensorproduct}
The subcategory $\SH(S)^{\Veff}$ of $\SH(S)$ is closed under the smash product.
\end{lemma}
\begin{proof}
To begin with, 
suppose $\E\in\SH(S)^{\Veff}$ and $X\in\Sm$.
Then $\Sigma^\infty X_{+}\wedge\E$ lies in $\SH(S)^{\Veff}$ by the following ``induction'' argument on the form of $\E$.
It clearly holds when $\E=\Sigma^\infty Y_{+}$ for some $Y\in\Sm$. 
Suppose $\E=\hocolim\,\E_i$ and $\Sigma^\infty X_{+}\wedge \E_i \in \SH(S)^{\Veff}$. 
Then $\Sigma^\infty X_{+}\wedge \E \in \SH(S)^{\Veff}$ because $\SH(S)^{\Veff}$ is closed under homotopy colimits. 
Furthermore, 
if in a triangle
$$
\AAA \to \E \to \B \to \AAA[1], 
$$ 
$\Sigma^\infty X_{+}\wedge \AAA \in \SH(S)^{\Veff}$ and likewise for $\B$, 
then $\Sigma^\infty X_{+}\wedge\E \in\SH(S)^{\Veff}$ because $\SH(S)^{\Veff}$ is closed under extensions by definition.
A similar ``induction'' argument in the first variable shows now that for all objects $\F,\E \in\SH(S)^{\Veff}$ the smash product 
$\F \wedge \E \in\SH(S)^{\Veff}$.
\end{proof}

For the definition of the algebraic cobordism spectrum $\MGL$ we refer to \cite{Voevodsky-icm}.
One of the reasons why the category $\SH(S)^{\Veff}$ is of interest is that it contains $\MGL$ for general base schemes.
\begin{theorem}
\label{theorem:MGLveff}
The algebraic cobordism spectrum $\MGL$ is very effective.
\end{theorem}

In fact our proof of Theorem \ref{theorem:MGLveff} shows the following stronger statement:
The cofiber of the unit map $\unit \to \MGL$ is contained in $\Sigma_T \SH(S)^{\Delta}_{\geq 0}$, 
where $\SH(S)^{\Delta}_{\geq 0}$ is the smallest full saturated subcategory of $\SH(S)$ that contains the suspension spectra 
$\Sigma^{2i,i} \unit$ for every $i\ge 0$ and is closed under homotopy colimits and extensions.
The notation $\Sigma_T$ refers to suspension with respect to the Tate object,
i.e.~$\Sigma_T=\Sigma^{2,1}$ in the usual bigrading.
\begin{lemma} 
\label{lemma:cof}
Let $r$ be an integer and suppose
$$
\Sigma^{2r,r} \unit \to \AAA \to \B \to \Sigma^{2r+1,r} \unit
$$ 
and
$$
\AAA \to \E \to \F \to \AAA[1]
$$
are triangles in $\SH(S)$. 
If $\B,\F\in\Sigma_T^{r+1}\SH(S)^{\Delta}_{\geq 0}$ then the cofiber of $\Sigma^{2r,r}\unit\to\E$ lies in $\Sigma_T^{r+1}\SH(S)^{\Delta}_{\geq 0}$.
\end{lemma}
\begin{proof}
This follows since the cofiber of $\Sigma^{2r,r}\unit\to\E$ is an extension of $\F$ by $\B$ and the category 
$\Sigma_T^{r+1} \SH(S)^{\Delta}_{\geq 0}$ is closed under extensions.
\end{proof}

Let $\GG(n,d)$ denote the Grassmannian parametrizing locally free quotients of rank $d$ of the trivial bundle of rank $n$. 
Recall there is a universal subsheaf $\caK_{n,d}$ of $\caO^{n}$ and a natural map $\iota\colon\GG(n,d) \to \GG(n+1,d)$ that classifies the 
subbundle $\caK_{n,d}\oplus \caO$ of $\caO^{n+1}$. 
Denote by $\overline{\iota}$ the canonical point of $\GG(n,d)$ obtained by the composite map 
$$
\ast
\cong
\GG(d,d)
\overset{\iota}{\to}
\GG(d+1,d)
\overset{\iota}{\to}
\cdots
\overset{\iota}{\to}
\GG(n,d).
$$
We are interested in vector bundles of a particular type over Grassmannians.
\begin{proposition} 
\label{proposition:555}
Suppose $\caE$ is a vector bundle of rank $r$ over the Grassmannian $\GG(n,d)$ which is a finite sum of copies of $\caK_{n,d}$ and its dual 
$\caK_{n,d}'$ and $\caO$.
Then $\overline{\iota}^* \caE$ is canonically trivialized. 
Furthermore the cofiber of the map between the suspension spectra of Thom spaces $\Sigma^{2r,r} \unit \to \Sigma^\infty \Thom(\caE)$ lies in 
$\Sigma_T^{r+1} \SH(S)^{\Delta}_{\geq 0}$.
\end{proposition}
\begin{proof}
We outline an argument which is reminiscent of the one for \cite[Proposition 3.6]{Spitzweckslices}.
The first step of the proof consists of showing there is an exact triangle
$$
\Sigma^\infty\Thom(\iota^*\caE) \to 
\Sigma^\infty\Thom(\caE) \to
\Sigma^\infty\Thom(\caE_{\GG(n,d)} \oplus \caK_{n,d}') \to 
\Sigma^\infty\Thom(\iota^*\caE)[1]
$$
for the canonical map $\iota\colon\GG(n,d+1) \to \GG(n+1,d+1)$ (that classifies the subbundle $\caK_{n,d+1}\oplus \caO\subseteq \caO^{n+1}$).
By induction we deduce that the cofiber of the canonical map $\Sigma^{2r,r} \unit \to \Sigma^\infty \Thom(\iota^* \caE)$ lies in 
$\Sigma_T^{r+1} \SH(S)^{\Delta}_{\geq 0}$. 
Again by applying induction, 
it follows that $\Sigma^\infty\Thom(\caE_{\GG(n,d)} \oplus \caK_{n,d}')\in \Sigma_T^{r+j} \SH(S)^{\Delta}_{\geq 0}$, 
where $j=n-d>0$. 
The proposition follows now from Lemma \ref{lemma:cof}.
\end{proof}

Next we give a proof of Theorem \ref{theorem:MGLveff}.
\begin{proof}
We denote by $\xi_n=\colim_d \caK_{n+d,d}$ the universal vector bundle over the infinite Grassmannian $\BGL_n=\colim_d \GG(n+1,d)$, 
and write
$$
\MGL=\hocolim_n \Sigma^{-2n,-n} \Sigma^\infty \Thom(\xi_n)=\hocolim_{n,d} \Sigma^{-2n,-n} \Sigma^\infty \Thom(\caK_{n+d,d}).
$$
The unit map $\unit \to \MGL$ is in turn induced by the maps 
$$
\Sigma^{-2n,-n} \Sigma^\infty\Thom(\overline{\iota}^*\caK_{n+d,d})\to \Sigma^{-2n,-n}\Sigma^\infty \Thom(\caK_{n+d,d}).
$$ 
By Proposition \ref{proposition:555} the cofibers of these maps are contained in $\Sigma_T \SH(S)^{\Delta}_{\geq 0}$. 
Since cofiber sequences are compatible with homotopy colimits,
this finishes the proof.
\end{proof}

\begin{lemma}
\label{lemma:veffvanishing}
Let $\E\in\SH(S)^{\Veff}$ and suppose $S$ is the spectrum of a perfect field. 
Then the homotopy group $\pi_{p,q}(\E)=0$ for $p<q$.
\end{lemma}
\begin{proof}
For suspension spectra of smooth projective schemes of finite type the claimed vanishing is stated in \cite[\S5.3]{Morelpi0}.
Suppose $\E=\hocolim\,\E_i$ where every $\E_i$ satisfies the conclusion of the lemma.
Minor variations of \cite[Corollary 4.4.2.4, Proposition 4.4.2.6]{Lurie} allows us to assuming the homotopy colimit is either a coproduct or a homotopy pushout.
For coproducts the result is clear, while for homotopy colimits the corresponding long exact sequence of homotopy sheaves implies the vanishing.
For a general extension $\AAA \to \E \to \B \to \AAA[1]$, 
where the vanishing holds for $\AAA$ and $\B$, 
the corresponding long exact sequence of homotopy groups implies the result.
\end{proof}

We denote by $\SH(S)^{\proj}$ the full thick subcategory of $\SH(S)$ generated by the objects $\Sigma^i_{T}\Sigma^\infty X_{+}$ for $X \in\Sm$ a projective scheme and $i\in\Z$.
\begin{proposition}
\label{proposition:strongconvergence}
Suppose the base scheme $S$ is a perfect field.
Let 
$$
\cdots \to\E_{i+1} \to \E_{i} \to  \E_{i-1}  \to  \cdots  \to  \E
$$ 
be a tower of motivic spectra such that $ \hocolim\,\E_{i}= \E$ and denote the corresponding filtration quotients by $\QQ_{i}$.
Suppose that $\E_{i}\in\Sigma_T^i \SH(S)^{\Veff}$ and $X \in \SH(S)^{\proj}$. 
If for each fixed $n$ the groups $\Hom(X,\E_{i}[n])$ stabilize as $i$ tends to minus infinity, 
then the spectral sequence of the tower with $E_2$-term $\Hom(X,\QQ_{\ast}[\ast])$ and target graded group $\Hom(X,\E[\ast])$ converges strongly.
\end{proposition}
\begin{proof}
Smashing the tower with the Spanier-Whitehead dual $D(X)$ of $X$ produces a tower with terms $D(X)\wedge\E_{i}\in \Sigma_T^{i+n} \SH(S)^{\Veff}$ for a fixed integer $n$.
Hence we may assume that $X$ is the sphere spectrum because smooth projective schemes of finite type over $S$ are dualizable \cite{Hu}.
The spectral sequence obtained from the exact couple associated to the tower is strongly convergent due to Lemma \ref{lemma:veffvanishing}.
\end{proof}

For the complex cobordism spectrum $\MU$, 
fix an isomorphism $\MU_{\ast}\cong\Z[\x_{1},\x_{2},\dots]$ where $\vert \x_{i} \vert =i$ and consider the canonical map $\MU_{\ast}\to\MGL_{\ast}$.

\begin{proposition}
\label{proposition:kglisoquotient}
Over fields of characteristic zero there is a natural isomorphism from the quotient of $\MGL$ by the sequence $(\x_{i})_{i\geq 2}$ to $\kgl=f_{0}\KGL$.
\end{proposition}
\begin{proof}
The orientation map $\MGL\to\KGL$ sends $\x_{i}\in\MGL_{2i,i}$ to $0$ in $\KGL_{2i,i}$ for $i\geq 2$.
Hence there is a naturally induced map from the quotient of $\MGL$ by the sequence $(\x_{i})_{i\geq 2}$ to $\KGL$.
Since the quotient $\MGL/(\x_{i})_{i\geq 2}$ is an effective spectrum we obtain the desired map to $\kgl$. 
As shown in \cite[Proposition 5.4]{Spitzweckslices} this map induces an isomorphism on all slices.
(The proof employs the work of Hopkins-Morel on quotients of $\MGL$.)
For any $X$ of $\SH(S)^{\proj}$ we may consider the spectral sequences obtained by taking homs into the respective slice filtrations of $\kgl$ and the quotient.
Theorem \ref{theorem:MGLveff} and Proposition \ref{proposition:strongconvergence} ensure that the spectral sequence for the quotient is strongly convergent.
For $\kgl$,
strong convergence holds by \cite[Proposition 5.5]{Voevodskymssktheory}.
(Note that Conjecture 4 in \cite{Voevodskymssktheory} is proven in \cite{Levine-slices}, 
cf.~the introduction in loc.~cit.~for a discussion.)
Our claim follows now by comparing the target graded groups of these spectral sequences.
\end{proof}

\begin{corollary} \label{corollary:kglveryeffective}
Over fields of characteristic zero the connective K-theory spectrum $\kgl$ is very effective.
\end{corollary}
\begin{proof}
Combine Theorem \ref{theorem:MGLveff} and Proposition \ref{proposition:kglisoquotient} with the fact that very effectiveness is preserved under homotopy colimits.
\end{proof}

\begin{lemma}
\label{lemma:veryeffective}
The motivic spectrum $\Sigma^\infty \X^\tau_{+} \wedge_{\Sigma^\infty \PP^\infty_+}\kgl$ is very effective.
\end{lemma}
\begin{proof}
For $n\geq 0$, 
Corollary \ref{corollary:kglveryeffective} shows the smash product $\Sigma^{\infty}\X^{\tau}_{+} \wedge (\Sigma^{\infty} \PP^{\infty}_{+})^{\wedge n}\wedge\kgl$ is very effective 
since $\SH(S)^{\Veff}$ is closed under smash products in $\SH(S)$ according to Lemma \ref{lemma:veffclosedtensorproduct}.
When $n$ varies, 
\begin{equation*}
n\mapsto
\Sigma^{\infty}\X^{\tau}_{+} \wedge (\Sigma^{\infty} \PP^{\infty}_{+})^{\wedge n}\wedge\kgl
\end{equation*}
defines a simplicial object in motivic symmetric spectra.
Its homotopy colimit is very effective and identifies with the smash product $\Sigma^\infty \X^\tau_{+} \wedge_{\Sigma^\infty \PP^\infty_+}\kgl$.
\end{proof}

We denote by $\Ho(\Sigma^\infty \PP^\infty_+ - \Mod)^\proj$ the full thick subcategory of the homotopy category $\Ho(\Sigma^\infty \PP^\infty_+ - \Mod)$ generated by the 
objects $\Sigma^i_T\Sigma^\infty \PP^\infty_+ \wedge\Sigma^\infty X_+$ for $X \in \Sm$ projective and $i\in\Z$.
\begin{lemma}
\label{lemma:veryeffective2}
Suppose $\Sigma^\infty \X^\tau_+$ is an object of $\Ho(\Sigma^\infty \PP^\infty_+ - \Mod)^\proj$.
Then there exists an integer $n\in\Z$ such that $\uHom_{\Sigma^\infty \PP^\infty_+}(\Sigma^\infty \X^\tau_+,\kgl)$ lies in $\Sigma_T^n \SH(S)^\Veff$.
\end{lemma}
\begin{proof}
Note first that 
for every $X \in \Sm$, $i,j \in \Z$, 
there is an $n_{1} \in \Z$ such that $\Sigma^{i,j} \Sigma^\infty X_+ \in \Sigma_T^{n_{1}} \SH(S)^\Veff$.
For $X$ projective we get that $D(\Sigma^\infty X_+) \in \Sigma_T^{n_{2}} \SH(S)^\Veff$ for some $n_{2}\in\Z$ by \cite[Appendix]{Hu} 
and we conclude that 
$$
\uHom_{\Sigma^\infty \PP^\infty_+}(\Sigma^{i,j}\Sigma^\infty \PP^\infty_+ \wedge\Sigma^\infty X_+,\kgl)
\cong 
\Sigma^{-i,-j} D(\Sigma^\infty X_+) \wedge \kgl
$$ 
lies in $\Sigma_T^{n_{3}} \SH(S)^\Veff$ for some $n_{3}\in\Z$ by Corollary \ref{corollary:kglveryeffective}.
This shows the result for all of the generators of $\Ho(\Sigma^\infty \PP^\infty_+ - \Mod)^\proj$.
The general case follows routinely by taking cones and direct summands.
\end{proof}

\begin{theorem}
\label{theorem:convergencetwistedKss}
Let $S$ be a field of characteristic zero.
Suppose $\Sigma^\infty \X_{+}$ is compact, equivalently strongly dualizable, equivalently an object of $\SH(S)^{\proj}$.
Then the motivic twisted K-theory spectral sequence 
\begin{equation*}
\MZ_{\ast}(\Sigma^\infty \X_{+} )
\Longrightarrow
\KGL^{\tau}_{\ast}(\X)
\end{equation*}
in (\ref{equation:twistedKss1}) is strongly convergent.
\end{theorem}
\begin{proof}
The proof follows by reference to Proposition \ref{proposition:strongconvergence} in the case when $X=S^{0,q}$.
Two assumptions need to be checked,
i.e.~$E_{i}=\Sigma^{2i,i}\Sigma^\infty \X_{+}\wedge_{\Sigma^{\infty} \PP^{\infty}_{+}}\kgl\in\Sigma^{i}_{\PP^{1}}\SH(S)^{\Veff}$ and stability.
Very effectiveness and Lemma \ref{lemma:veryeffective} verify that the first assumption holds.
Second,
the stabilization condition is equivalent to the fact that for a fixed $n$, 
the group $\pi_{n,q}\QQ_{i}(\X^{\tau})\neq 0$ for only finitely many $i$. 
The latter follows from the corresponding statement in the untwisted case because $\Sigma^\infty \X_{+}$ is strongly dualizable.
Namely, 
letting $i \ll 0$ vary, 
the groups $\Hom(\Sigma^{n,q} D(\Sigma^\infty \X_{+}),f_{i}\KGL)$ become isomorphic.
\end{proof}

\begin{theorem}
Let $S$ be a field of characteristic zero.
Suppose $\Sigma^\infty \X^\tau_{+}$ is compact in $\Ho(\Sigma^\infty \PP^\infty_+ - \Mod)$,
equivalently strongly dualizable.
Then the motivic twisted K-theory spectral sequence 
\begin{equation*}
\MZ^{\ast}(\Sigma^\infty \X_{+} )
\Longrightarrow
\KGL_{\tau}^{\ast}(\X)
\end{equation*}
in (\ref{equation:twistedKss2}) is strongly convergent.
\end{theorem}
\begin{proof} We first note that the assumptions imply that $\Sigma^\infty \X_+$ is strongly dualizable by Lemma \ref{lemma:P00topoint}.
The proof proceeds now along the lines of the proof of Theorem \ref{theorem:convergencetwistedKss} with a reference to Lemma \ref{lemma:veryeffective2} for very effectiveness.
\end{proof}

We end this section by discussing the closely related approach of the slice spectral sequence for $\KGL^\tau$.
Recall that the slices of any motivic spectrum fit into the slice tower constructed by Voevodsky in \cite{Voevodsky-slices}.

An identification of the zero-slice of $\KGL^{\tau}$ would in turn determine all the slices $s_{n}\KGL^{\tau}$ by 
$(2,1)$-periodicity, 
i.e.~there is an isomorphism in the motivic stable homotopy category
\begin{equation*}
s_{n}\KGL^{\tau}
\cong
\Sigma^{2n,n}s_{0}\KGL^{\tau}.
\end{equation*}
This follows from the evident $\KGL$-module structure on $\KGL^{\tau}$ and the Bott periodicity isomorphism $\beta\colon\Sigma^{2,1}\KGL\rto\KGL$ furnishing the composite isomorphism 
\begin{equation*}
\Sigma^{2,1}\KGL^{\tau}\rto
\Sigma^{2,1}\KGL^{\tau}\wedge\KGL\rto
\KGL^{\tau}\wedge\Sigma^{2,1}\KGL\rto
\KGL^{\tau}\wedge\KGL\rto
\KGL^{\tau}.
\end{equation*}
The same comments apply to $\KGL_{\tau}$.
If the base scheme is a perfect field, 
then all of the slices $s_{n}\KGL^{\tau}$ and $s_{n}\KGL_{\tau}$ are in fact motives, 
i.e.~modules over the integral motivic Eilenberg-MacLane spectrum $\MZ$, 
cf.~\cite{Levine-slices}, \cite{Pelaez}, \cite{Rondigs-Ostvar1}, \cite{Rondigs-Ostvar2}, \cite{Voevodsky-zeroslices}, \cite{VRO}.
However,
except for example when $\X$ is the point and $\tau$ the trivial twist, 
the slice spectral sequences cannot coincide with the spectral sequence constructed earlier in this section.
Indeed the corresponding filtration quotients are different because by weight considerations the smash product $\X \wedge \MZ$ is not a zero slice in general.

\section{Further problems and questions}
\label{section:furtherquestionsandproblems}
We end the main body of the paper by discussing some problems and questions related to motivic twisted K-theory.

A pressing question left open in the previous section is to identify the $d_{1}$-differentials in the spectral 
sequences.
\begin{problem}
\label{problem:differentials}
Express the $d_{1}$-differentials in the slice spectral sequence for $\KGL_{\tau}$ in terms of motivic Steenrod squares and the twist $\tau$.
\end{problem}
\begin{remark}
The $d_{3}$-differentials in the Atiyah-Hirzebruch spectral sequence for twisted K-theory were identified by Atiyah and Segal in \cite{Atiyah-Segal2} as the difference 
between $\Sq^{3}$ and the twisting.
In the same paper the higher differentials are determined in terms of Massey products.
One may ask if also the higher differentials in the slice spectral sequence can be described in terms of Massey products.
\end{remark}

\begin{problem}
\label{problem:products}
For twists $\tau$ and $\tau^{\prime}$ construct products 
$$
\KGL_{\tau}
\wedge
\KGL_{\tau^{\prime}}
\to
\KGL_{\tau+\tau^{\prime}}
$$
and investigate its properties.
\end{problem}

In Remark \ref{remark:BBGtrivial} we noted that all the motivic twisted K-groups of the identity map of $\B\BG$ are trivial.
More generally,
if $\tau$ is any twisting of $\B\BG$ one may ask if the motivic $\tau$-twisted K-groups are trivial.
This is the content of the next problem asking when $\B\BG$ is a point for motivic twisted K-theory.
\begin{problem}
\label{problem:point}
For which twists of $\B\BG$ is the associated motivic twisted K-theory trivial? 
\end{problem}
\begin{remark}
The corresponding problem in topology has an affirmative solution for all twists by work of Anderson and Hodgkin \cite{Anderson-Hodgkin}.
By analogy,
work on Problem \ref{problem:point} is likely to involve a computation of the $\KGL$-homology of the motivic Eilenberg-MacLane spaces $K(\Z/n,2)$ for $n\geq 1$ any integer.
\end{remark}
Twisted equivariant K-theory for compact Lie groups is closely related to loop groups \cite{FHT}.
It is natural to ask for a generalization of our construction of motivic twisted K-theory to an equivariant setting involving group schemes.  
\begin{problem}
\label{problem:equivariant}
Develop a theory of motivic twisted equivariant K-theory.
\end{problem}

The last problem we suggest is a very basic one.
The construction of motivic twisted K-theory should generalize to other examples.
We wish to single out hermitian K-theory as a closely related example of much interest.
In this case we expect the twistings arise from homotopy classes of maps from $\X$ to the classifying space $\B\B\mu_{2}$, 
i.e.~elements of the second mod-$2$ motivic cohomology group $\MZ^{2,1}(\X;\Z/2)$ of weight one.
\begin{problem}
\label{problem:hermitian}
Develop a theory of motivic twisted hermitian K-theory.
\end{problem}

\section{Graded Adams Hopf algebroids}
\label{section:hopfalgebroids}
Recall that a Hopf algebroid is a cogroupoid object in the category of commutative rings \cite[Appendix A1]{Ravenel}.
Let $(A,\Gamma)$ be a Hopf algebroid.
If the left unit $\eta_{L}\colon A\rto\Gamma$ classifying the domain is flat, 
or equivalently the right unit $\eta_{R}\colon A\rto\Gamma$ classifying the codomain is flat, 
then $(A,\Gamma)$ is called a flat Hopf algebroid. 
If $A\rto B$ is a ring map, 
we write $B\otimes_{A}\Gamma$ for the tensor product when $\Gamma$ is given an $A$-module structure via $\eta_{L}$ and 
$\Gamma\otimes_{A}B$ when $\Gamma$ is given an $A$-module structure via $\eta_{R}$.
An $(A,\Gamma)$-comodule comprises an $A$-module $M$ together with a coassociative and unital map of left $A$-modules $M\rto\Gamma\otimes_{A}M$ 
(see e.g.~\cite[Appendix A1]{Ravenel}).
The category of $(A,\Gamma)$-comodules with the evident notion of a morphism is an abelian category provided $\Gamma$ is a flat right $A$-module via $\eta_{R}$.

Likewise, 
a graded Hopf algebroid is  a cogroupoid object in the category of graded commutative rings \cite[Appendix A1]{Ravenel}.
The notions of flat graded Hopf algebroids and comodules over a graded Hopf algebroid are defined exactly as in the ungraded setting. 

The examples of Hopf algebroids of main interest in stable homotopy theory are so-called ``Adams Hopf algebroids.'' 
In the graded setting we make the following definition:
A graded Hopf algebroid $(A,\Gamma)$ is called a graded Adams Hopf algebroid if $\Gamma$ is the colimit of a filtered system of graded comodules which are finitely generated 
and projective as graded $A$-modules.
\begin{proposition}
\label{proposition:HopfalgebroidstructureKGL}
The pair $(\KGL_{\ast},\KGL_{\ast}\KGL)$ is a flat graded Adams Hopf algebroid.
\end{proposition}
\begin{proof}
We give two proofs of this result.

Since homology commutes with sequential colimits we get that $\KGL_{\ast}(\PP^{\infty})$ is a filtered colimit of comodules which are finitely generated free $\KGL_{\ast}$-modules.
Hence the same holds for $\KGL_{\ast}\KGL$ by using the Bott tower (\ref{equation:Botttower}) for $\Sigma^{\infty}\PP^{\infty}_{+}$ as a model for $\KGL$. 

For the second proof,
recall the base change isomorphism in (\ref{equation:secondisomorphism}), 
\begin{equation*}
\KGL_{\ast}\KGL
\cong
\KGL_{\ast}\otimes_{\KU_{\ast}}\KU_{\ast}\KU.
\end{equation*}
By the topological analogue of the first proof we see that $(\KU_{\ast},\KU_{\ast}\KU)$ is a flat Adams Hopf algebroid.
We conclude by pulling back the filtered colimit to the tensor product.
\end{proof}

\begin{proposition}
\label{proposition:hopfcomodulealgebra}
Let $(A,\Gamma)$ be a graded Hopf algebroid and $A\rto B$ a graded ring map.
Suppose $B$ is a graded $(A,\Gamma)$-comodule algebra.
\begin{itemize}
\item[(i)]
The pair $(B,B\otimes_{A}\Gamma)$ is a graded Hopf algebroid.
\item[(ii)]
If $(A,\Gamma)$  is flat, 
then so is $(B,B\otimes_{A}\Gamma)$.
\item[(iii)]
If $(A,\Gamma)$  is a graded Adams Hopf algebroid, 
then so is $(B,B\otimes_{A}\Gamma)$.
\end{itemize}
\end{proposition}
\begin{proof}
For $C$ be a graded (commutative) algebra,
let $X=\Hom(A,C)$, $M=\Hom(\Gamma,C)$ and $Y=\Hom(B,C)$. 
Then $(X,M)$ is a groupoid and $Y$ is a set over $X$ equipped with an $M$-action. 
It is easily seen that the pair $(Y,Y \times_X M)$ acquires the structure of a groupoid. 
This settles the first part.

The second part follows by a standard base change argument,
while the third point follows by pulling back the graded sub-comodules which are finitely generated projective as graded $A$-modules.
\end{proof}

The graded Hopf algebroid of primary interest in this paper is the following example.
\begin{proposition}
\label{proposition:Hopfalgebroidstructure}
The pair 
\begin{equation*}
(\KGL_{\ast}(\PP^{\infty}),\KGL_{\ast}\KGL\otimes_{\KGL_{\ast}}\KGL_{\ast}(\PP^{\infty}))
\end{equation*}
is a flat graded Adams Hopf algebroid.
\end{proposition}
\begin{proof}
Note first that $\KGL_{\ast}(\PP^{\infty})$ has the structure of  graded comodule algebra over $(\KU_{\ast},\KU_{\ast}\KU)$. 
The result follows from Propositions \ref{proposition:HopfalgebroidstructureKGL} and \ref{proposition:hopfcomodulealgebra}.
\end{proof}

\begin{remark}
\label{remark:hopfalgebroidremark}
The left unit map denoted by $\eta_{\KGL_{\ast}(\PP^{\infty})}$ is determined by the composite map
\begin{equation*}
\KGL\wedge\Sigma^{\infty}\PP^{\infty}_{+}
\cong
\KGL\wedge\unit\wedge\Sigma^{\infty}\PP^{\infty}_{+}
\rto
\KGL\wedge\KGL\wedge\Sigma^{\infty}\PP^{\infty}_{+}
\end{equation*}
and the right unit map by 
\begin{equation*}
\KGL\wedge\Sigma^{\infty}\PP^{\infty}_{+}
\cong
\unit\wedge\KGL\wedge\Sigma^{\infty}\PP^{\infty}_{+}
\rto
\KGL\wedge\KGL\wedge\Sigma^{\infty}\PP^{\infty}_{+}.
\end{equation*}
(Here we make use of the unit map from the motivic sphere spectrum $\unit$ to $\KGL$.)

We use the isomorphism
\begin{equation*}
\KGL_{\ast}(\KGL\wedge\Sigma^{\infty}\PP^{\infty}_{+})
\cong
\KGL_{\ast}\KGL
\otimes_{\KGL_{\ast}}
\KGL_{\ast}(\PP^{\infty}).
\end{equation*}

Passing to $\KGL$-homology under the left unit map yields a map 
\begin{equation*}
\KGL_{\ast}(\PP^{\infty})
\rto
\KGL_{\ast}\KGL
\otimes_{\KGL_{\ast}}
\KGL_{\ast}(\PP^{\infty})
\end{equation*}
displaying $\KGL_{\ast}(\PP^{\infty})$ as a comodule over $\KGL_{\ast}\KGL$.

The augmentation is determined by the multiplication on $\KGL$ via the map
\begin{equation*}
\KGL\wedge\KGL\wedge\Sigma^{\infty}\PP^{\infty}_{+}
\rto
\KGL\wedge\Sigma^{\infty}\PP^{\infty}_{+}.
\end{equation*}
\end{remark}

The following is the graded version of the notion of Landweber exactness introduced by Hovey and Strickland in \cite[Definition 2.1]{Hovey-Strickland}.
\begin{definition}
Suppose $(A,\Gamma)$ is a flat graded Hopf algebroid.
Then a graded ring map $A\rto B$ is Landweber exact over $(A,\Gamma)$ if the functor $-\otimes_{A}B$ from graded $(A,\Gamma)$-comodules to graded $B$-modules is exact.
\end{definition}

By abuse of notation we let $\eta_{L}$ denote the composite map 
\begin{equation*}
A
\overset{\eta_{L}}{\rto}
\Gamma
\cong
\Gamma\otimes_{A}A
\rto
\Gamma\otimes_{A}B.
\end{equation*}
The next lemma is well known. 
For the convenience of the reader we shall sketch a proof since the result is employed in the proof of Theorem \ref{theorem:landweberexactness}.
\begin{lemma}
\label{lemma:landweberexctnessflatness}
Suppose $(A,\Gamma)$ is a flat graded Hopf algebroid.
Then a graded ring map $A\rto B$ is Landweber exact over $(A,\Gamma)$ if and only if the map 
$\eta_{L}\colon A\rto\Gamma\otimes_{A}B$ is flat.
\end{lemma}
\begin{proof}
The only if implication holds because $\Gamma\otimes_{A}-$ preserves monomorphisms between graded $A$-modules.  
Conversely, 
for every graded $A$-comodule $M$, 
the graded coaction map $M\rto\Gamma\otimes_{A}M$ is a retract.  
Thus for a monomorphism of graded comodules $M\rto N$ the map $B\otimes_{A}M\rto B\otimes_{A}N$ is a retract of 
$B\otimes_{A}\Gamma\otimes_{A}M\rto B\otimes_{A}\Gamma\otimes_{A}N$.
\end{proof}
\begin{remark}
In the proof of Theorem \ref{theorem:edgemapisomorphism} we could have worked with the Hopf algebroid $(\KU_{\ast}(\CP^{\infty}),\KU_{\ast}\KU\otimes_{\KU_{\ast}}\KU_{\ast}(\CP^{\infty}))$ 
by restricting the comodule structure and using (ungraded) Landweber exactness.
In this way one can bootstrap a proof of Theorem \ref{theorem:edgemapisomorphism} more directly from \cite{Khorami} by using base change isomorphisms with no mention of graded Hopf algebroids.
In the same spirit,
we note there is an isomorphism
\begin{equation*}
\KGL_{\ast}^{\tau}(\X)
\cong
\KGL_{\ast}(\X^{\tau})\otimes_{\KU_{\ast}(\CP^{\infty})}\KU_{\ast}.
\end{equation*}
\end{remark}

{\bf Acknowledgements.}
We gratefully acknowledge hospitality and support from IMUB at the Universitat de Barcelona in the framework of the NILS mobility project. 
The second author benefitted from the hospitality of IAS in Princeton and TIFR in Mumbai during the preparation of this paper.
Both authors are partially supported by the RCN 185335/V30.

\bibliographystyle{plain} 
\bibliography{motivictwisted}

\begin{thebibliography}{10}

\bibitem{ArbeitMFO}
Arbeitsgemeinschaft mit aktuellem {T}hema: {T}wisted {$K$}-theory.
\newblock {\em Oberwolfach Rep.}, 3(4):2757--2804, 2006.
\newblock Abstracts from the workshop held October 8--14, 2006, Organized by U.
  Bunke, D. Freed and T. Schick.

\bibitem{Adams}
J.~F. Adams.
\newblock {\em Stable homotopy and generalised homology}.
\newblock University of Chicago Press, Chicago, Ill., 1974.
\newblock Chicago Lectures in Mathematics.

\bibitem{Anderson-Hodgkin}
D.~W. Anderson and L.~Hodgkin.
\newblock The {$K$}-theory of {E}ilenberg-{M}ac{L}ane complexes.
\newblock {\em Topology}, 7:317--329, 1968.

\bibitem{ABG}
M.~Ando, A.~J. Blumberg, and D.~Gepner.
\newblock Twists of {$K$}-theory and {$TMF$}, arxiv:1002.3004.

\bibitem{Atiyah-Segal1}
M.~Atiyah and G.~Segal.
\newblock Twisted {$K$}-theory.
\newblock {\em Ukr. Mat. Visn.}, 1(3):287--330, 2004.

\bibitem{Atiyah-Segal2}
M.~Atiyah and G.~Segal.
\newblock Twisted {$K$}-theory and cohomology.
\newblock In {\em Inspired by {S}. {S}. {C}hern}, volume~11 of {\em Nankai
  Tracts Math.}, pages 5--43. World Sci. Publ., Hackensack, NJ, 2006.

\bibitem{BL}
S.~Bloch and S.~Lichtenbaum.
\newblock A spectral sequence for motivic cohomology, {K}-theory {P}reprint
  {A}rchives:0062.

\bibitem{BCMMS}
P.~Bouwknegt, A.~L. Carey, V.~Mathai, M.~K. Murray, and D.~Stevenson.
\newblock Twisted {$K$}-theory and {$K$}-theory of bundle gerbes.
\newblock {\em Comm. Math. Phys.}, 228(1):17--45, 2002.

\bibitem{Cisinski-Deglise}
D.~C. Cisinski and F.~D{\'e}glise.
\newblock Triangulated categories of mixed motives, arxiv:0912.2110.

\bibitem{Karoubi-Donovan}
P.~Donovan and M.~Karoubi.
\newblock Graded {B}rauer groups and {$K$}-theory with local coefficients.
\newblock {\em Inst. Hautes \'Etudes Sci. Publ. Math.}, 38:5--25, 1970.

\bibitem{Dugger-Isaksen}
D.~Dugger and D.~C. Isaksen.
\newblock Motivic cell structures.
\newblock {\em Algebr. Geom. Topol.}, 5:615--652 (electronic), 2005.

\bibitem{DLORV}
B.~I. Dundas, M.~Levine, P.~A. {\O}stv{\ae}r, O.~R{\&quot;o}ndigs, and
  V.~Voevodsky.
\newblock {\em Motivic homotopy theory}.
\newblock Universitext. Springer-Verlag, Berlin, 2007.
\newblock Lectures from the Summer School held in Nordfjordeid, August 2002.

\bibitem{DRO}
B.~I. Dundas, O.~R{\&quot;o}ndigs, and P.~A. {\O}stv{\ae}r.
\newblock Motivic functors.
\newblock {\em Doc. Math.}, 8:489--525 (electronic), 2003.

\bibitem{FHT}
D.~S. Freed, M.~J. Hopkins, and T.~Teleman.
\newblock Loop groups and twisted {K}-theory {I}, arxiv:0711.1906.

\bibitem{FS}
E.~M. Friedlander and A.~Suslin.
\newblock The spectral sequence relating algebraic {$K$}-theory to motivic
  cohomology.
\newblock {\em Ann. Sci. \'Ecole Norm. Sup. (4)}, 35(6):773--875, 2002.

\bibitem{Gepner-Snaith}
D.~Gepner and V.~Snaith.
\newblock On the motivic spectra representing algebraic cobordism and algebraic
  {$K$}-theory.
\newblock {\em Doc. Math.}, 14:359--396 (electronic), 2009.

\bibitem{Grayson}
D.~R. Grayson.
\newblock The motivic spectral sequence.
\newblock In {\em Handbook of {$K$}-theory. {V}ol. 1, 2}, pages 39--69.
  Springer, Berlin, 2005.

\bibitem{EGAIV}
A.~Grothendieck.
\newblock \'{E}l\'ements de g\'eom\'etrie alg\'ebrique. {IV}. \'{E}tude locale
  des sch\'emas et des morphismes de sch\'emas {IV}.
\newblock {\em Inst. Hautes \'Etudes Sci. Publ. Math.}, 32:361, 1967.

\bibitem{GSO}
J.~J. Guti{\'e}rrez, O.~R{\&quot;o}ndigs, M.~Spitzweck, and P.~A.
  {\O}stv{\ae}r.
\newblock Slices and colored operads.

\bibitem{Hornbostel}
J.~Hornbostel.
\newblock Preorientations of the derived motivic multiplicative group,
  arxiv:1005.4546.

\bibitem{Hovey}
M.~Hovey.
\newblock Homotopy theory of comodules over a {H}opf algebroid.
\newblock In {\em Homotopy theory: relations with algebraic geometry, group
  cohomology, and algebraic {$K$}-theory}, volume 346 of {\em Contemp. Math.},
  pages 261--304. Amer. Math. Soc., Providence, RI, 2004.

\bibitem{Hovey-Strickland}
M.~Hovey and N.~Strickland.
\newblock Comodules and {L}andweber exact homology theories.
\newblock {\em Adv. Math.}, 192(2):427--456, 2005.

\bibitem{Hu}
P.~Hu.
\newblock On the {P}icard group of the stable {$\Bbb A^1$}-homotopy category.
\newblock {\em Topology}, 44(3):609--640, 2005.

\bibitem{Jardine}
J.~F. Jardine.
\newblock Motivic symmetric spectra.
\newblock {\em Doc. Math.}, 5:445--553 (electronic), 2000.

\bibitem{Karoubi}
M.~Karoubi.
\newblock Twisted {$K$}-theory---old and new.
\newblock In {\em {$K$}-theory and noncommutative geometry}, EMS Ser. Congr.
  Rep., pages 117--149. Eur. Math. Soc., Z\&quot;urich, 2008.

\bibitem{Khorami}
M.~Khorami.
\newblock A universal coefficient theorem for twisted {K}-theory,
  arxiv:1001.4790.

\bibitem{Levine-slices}
M.~Levine.
\newblock The homotopy coniveau tower.
\newblock {\em J. Topol.}, 1(1):217--267, 2008.

\bibitem{Lurie}
J.~Lurie.
\newblock {\em Higher topos theory}, volume 170 of {\em Annals of Mathematics
  Studies}.
\newblock Princeton University Press, Princeton, NJ, 2009.

\bibitem{MS}
J.~P. May and J.~Sigurdsson.
\newblock {\em Parametrized homotopy theory}, volume 132 of {\em Mathematical
  Surveys and Monographs}.
\newblock American Mathematical Society, Providence, RI, 2006.

\bibitem{Morelpi0}
F.~Morel.
\newblock On the motivic {$\pi_0$} of the sphere spectrum.
\newblock In {\em Axiomatic, enriched and motivic homotopy theory}, volume 131
  of {\em NATO Sci. Ser. II Math. Phys. Chem.}, pages 219--260. Kluwer Acad.
  Publ., Dordrecht, 2004.

\bibitem{nmp-nonregular}
N.~Naumann, M.~Spitzweck, and P.~A. {\O}stv{\ae}r.
\newblock Chern classes, {$K$}-theory and {L}andweber exactness over nonregular
  base schemes.
\newblock In {\em Motives and algebraic cycles}, volume~56 of {\em Fields Inst.
  Commun.}, pages 307--317. Amer. Math. Soc., Providence, RI, 2009.

\bibitem{nmp-exactness}
N.~Naumann, M.~Spitzweck, and P.~A. {\O}stv{\ae}r.
\newblock Motivic {L}andweber exactness.
\newblock {\em Doc. Math.}, 14:551--593, 2009.

\bibitem{Pelaez}
P.~Pelaez.
\newblock Multiplicative properties of the slice filtration, arxiv:0806.1704,
  to appear in {A}st{\'e}risque.

\bibitem{Quillen}
D.~Quillen.
\newblock On the cohomology and {$K$}-theory of the general linear groups over
  a finite field.
\newblock {\em Ann. of Math. (2)}, 96:552--586, 1972.

\bibitem{Ravenel}
D.~C. Ravenel.
\newblock {\em Complex cobordism and stable homotopy groups of spheres}, volume
  121 of {\em Pure and Applied Mathematics}.
\newblock Academic Press Inc., Orlando, FL, 1986.

\bibitem{Ravenel-Wilson}
D.~C. Ravenel and W.~S. Wilson.
\newblock The {H}opf ring for complex cobordism.
\newblock {\em J. Pure Appl. Algebra}, 9(3):241--280, 1976/77.

\bibitem{Rosenberg}
J.~Rosenberg.
\newblock Continuous-trace algebras from the bundle theoretic point of view.
\newblock {\em J. Austral. Math. Soc. Ser. A}, 47(3):368--381, 1989.

\bibitem{Rondigs-Ostvar1}
O.~R{\&quot;o}ndigs and P.~A. {\O}stv{\ae}r.
\newblock Motives and modules over motivic cohomology.
\newblock {\em C. R. Math. Acad. Sci. Paris}, 342(10):751--754, 2006.

\bibitem{Rondigs-Ostvar2}
O.~R{\&quot;o}ndigs and P.~A. {\O}stv{\ae}r.
\newblock Modules over motivic cohomology.
\newblock {\em Adv. Math.}, 219(2):689--727, 2008.

\bibitem{Rondigs-Spitzweck-Ostvar}
O.~R{\&quot;o}ndigs, M.~Spitzweck, and P.~A. {\O}stv{\ae}r.
\newblock Motivic strict ring models for {K}-theory.
\newblock {\em Proc. Amer. Math. Soc.}, 138(10):3509--3520, 2010.

\bibitem{schwede-shipley}
S.~Schwede and B.~E. Shipley.
\newblock Algebras and modules in monoidal model categories.
\newblock {\em Proc. London Math. Soc. (3)}, 80(2):491--511, 2000.

\bibitem{Spitzweckslices}
M.~Spitzweck.
\newblock Relations between slices and quotients of the algebraic cobordism
  spectrum, arxiv:0812.0749.

\bibitem{Spitzweck-Ostvar}
M.~Spitzweck and P.~A. {\O}stv{\ae}r.
\newblock The {B}ott inverted infinite projective space is homotopy algebraic
  {$K$}-theory.
\newblock {\em Bull. Lond. Math. Soc.}, 41(2):281--292, 2009.

\bibitem{Suslin}
A.~Suslin.
\newblock On the {G}rayson spectral sequence.
\newblock {\em Tr. Mat. Inst. Steklova}, 241(Teor. Chisel, Algebra i Algebr.
  Geom.):218--253, 2003.

\bibitem{Suslin-Voevodsky}
A.~Suslin and V.~Voevodsky.
\newblock Bloch-{K}ato conjecture and motivic cohomology with finite
  coefficients.
\newblock In {\em The arithmetic and geometry of algebraic cycles ({B}anff,
  {AB}, 1998)}, volume 548 of {\em NATO Sci. Ser. C Math. Phys. Sci.}, pages
  117--189. Kluwer Acad. Publ., Dordrecht, 2000.

\bibitem{TXLG}
J.-L. Tu, P.~Xu, and C.~Laurent-Gengoux.
\newblock Twisted {$K$}-theory of differentiable stacks.
\newblock {\em Ann. Sci. \'Ecole Norm. Sup. (4)}, 37(6):841--910, 2004.

\bibitem{Voevodsky-icm}
V.~Voevodsky.
\newblock {$\bold A\sp 1$}-homotopy theory.
\newblock In {\em Proceedings of the International Congress of Mathematicians,
  Vol. I (Berlin, 1998)}, volume Extra Vol. I, pages 579--604 (electronic),
  1998.

\bibitem{Voevodsky-slices}
V.~Voevodsky.
\newblock Open problems in the motivic stable homotopy theory. {I}.
\newblock In {\em Motives, polylogarithms and {H}odge theory, {P}art {I}
  ({I}rvine, {CA}, 1998)}, volume~3 of {\em Int. Press Lect. Ser.}, pages
  3--34. Int. Press, Somerville, MA, 2002.

\bibitem{Voevodskymssktheory}
V.~Voevodsky.
\newblock A possible new approach to the motivic spectral sequence for
  algebraic {$K$}-theory.
\newblock In {\em Recent progress in homotopy theory ({B}altimore, {MD},
  2000)}, volume 293 of {\em Contemp. Math.}, pages 371--379. Amer. Math. Soc.,
  Providence, RI, 2002.

\bibitem{Voevodsky-zeroslices}
V.~Voevodsky.
\newblock On the zero slice of the sphere spectrum.
\newblock {\em Tr. Mat. Inst. Steklova}, 246(Algebr. Geom. Metody, Svyazi i
  Prilozh.):106--115, 2004.

\bibitem{VRO}
V.~Voevodsky, O.~R{\&quot;o}ndigs, and P.~A. {\O}stv{\ae}r.
\newblock Voevodsky's {N}ordfjordeid lectures: motivic homotopy theory.
\newblock In {\em Motivic homotopy theory}, Universitext, pages 147--221.
  Springer, Berlin, 2007.

\bibitem{Witten}
E.~Witten.
\newblock D-branes and {$K$}-theory.
\newblock {\em J. High Energy Phys.}, 12:Paper 19, 41 pp.\ (electronic), 1998.

\end{thebibliography}
\vspace{0.1in}

\begin{center}
Department of Mathematics, University of Oslo, Norway.\\
e-mail: markussp@math.uio.no
\end{center}
\begin{center}
Department of Mathematics, University of Oslo, Norway.\\
e-mail: paularne@math.uio.no
\end{center}
\end{document}